\documentclass{article}

\usepackage[english]{babel}
\usepackage{cite,cancel,framed,marginnote}
\usepackage{amssymb}
\usepackage{amsmath,amsthm,tikz,multicol,bbm,tensor}
\usetikzlibrary{arrows,decorations.pathreplacing,decorations.markings}
\newcommand{\Hom}{\textnormal{Hom}}

\newcommand{\1}{\textbf{1}}

\DeclareMathOperator{\Vep}{Vec}

\DeclareMathOperator{\Tr}{Tr}
\DeclareMathOperator{\FP}{FPdim}

\theoremstyle{plain}
\newtheorem{theo}{Theorem}
\newtheorem{mytheorem}{Proposition}[subsection]
\newtheorem{mylemma}[mytheorem]{Lemma}
\newtheorem{mycorr}[mytheorem]{Corollary}
\theoremstyle{definition}
\newtheorem{example}[mytheorem]{Example}
\newtheorem*{note}{Note}
\newtheorem{defi}[mytheorem]{Definition}
\newtheorem*{theorem}{Theorem}

\title{Classification of $\mathfrak{sl}_3$ relations in the Witt group of nondegenerate braided fusion categories}
\date{\today}
\author{Andrew Schopieray\\ Mathematics Department, University of Oregon}

\begin{document}
\parskip = \baselineskip
\setlength{\parindent}{0cm}

\maketitle

\begin{abstract}
\noindent The Witt group of nondegenerate braided fusion categories $\mathcal{W}$ contains a subgroup $\mathcal{W}_\text{un}$ consisting of Witt equivalence classes of pseudo-unitary nondegenerate braided fusion categories.  For each finite-dimensional simple Lie algebra $\mathfrak{g}$ and positive integer $k$ there exists a pseudo-unitary category $\mathcal{C}(\mathfrak{g},k)$ consisting of highest weight integerable $\hat{g}$-modules of level $k$ where $\hat{\mathfrak{g}}$ is the corresponding affine Lie algebra.  Relations between the classes $[\mathcal{C}(\mathfrak{sl}_2,k)]$, $k\geq1$ have been completely described in the work of  Davydov, Nikshych, and Ostrik.  Here we give a complete classification of relations between the classes $[\mathcal{C}(\mathfrak{sl}_3,k)]$, $k\geq1$ with a view toward extending these methods to arbitrary simple finite dimensional Lie algebras $\mathfrak{g}$ and positive integer levels $k$.
\end{abstract}

\section{Introduction}

\par The Witt group of non-degenerate braided fusion categories $\mathcal{W}$ first introduced in \cite{DMNO} provides an algebraic structure that is one tool for organizing braided fusion categories.  Inside $\mathcal{W}$ lies the subgroup $\mathcal{W}_\text{un}$ consisting of classes of pseudo-unitary braided fusion categories which, in turn, contains the classes $[\mathcal{C}(\mathfrak{g},k)]$ coming from the representation theory of affine Lie algebras.  Theorem \ref{second} is the main goal of this paper, to classify all relations in the Witt group between the classes $[\mathcal{C}(\mathfrak{sl}_3,k)]$.  To do so requires identification of a unique (up to braided equivalence) representative of each Witt equivalence class which is simple and completely anisotropic (see Definitions \ref{def:simple} and \ref{def:anisotropic}), constructed in the cases where $3\mid k$ as the category of dyslectic $A$-modules $\mathcal{C}(\mathfrak{sl}_3,k)_A^0$ \cite{KiO}.  The major result which allows for the classification is Theorem 1 which states that the categories $\mathcal{C}(\mathfrak{sl}_3,k)_A^0$ are simple when $3\mid k$ and $k\neq3$.

\par Translated into the language of modular tensor categories, there is a common belief among physicists \cite{MS} that $\mathcal{W}_\text{un}$ is generated by the classes of the categories $\mathcal{C}(\mathfrak{g},k)$.  This provides at least one external motivation for understanding Witt group relations in $\mathcal{W}_\text{un}$.  But Witt group relations are difficult to come by; all relations in the subgroup $\mathcal{W}_\text{pt}\subset\mathcal{W}$ consisting of pointed braided fusion categories are known \cite{DGNO} and limited relations in $\mathcal{W}_\text{un}$ are known due to the theory of conformal embeddings of vertex operator algebras (Section \ref{sec:modcon}).  The general task of classifying all relations in $\mathcal{W}_\text{un}$ was presented in \cite{DMNO}, and in \cite{DNO} all relations among the classes of the categories $\mathcal{C}(\mathfrak{sl}_2,k)$ were classified.  Independent from the classification of Witt group relations, the passage from $\mathcal{C}\in\mathcal{W}_\text{un}$ to its category of dyslectic $A$-modules over a connected \'etale algebra has external connection to the process of anyon condensation in the physics literature, described for example in \cite[Chapter 6]{sebas} and \cite{kong}, providing stronger justification to conjecture and prove results similar to Theorem \ref{second} for general $\mathcal{C}(\mathfrak{g},k)$.  If these results are true they also provide an infinite collection of simple modular tensor categories which play an important role in the classification of all modular tensor categories, an open and active area of modern research.

\par The materials of this paper are organized as follows:  Section \ref{sec:prelims} contains a summary of the basic tools of modular tensor categories comparable to that found in \cite[Chapter 8]{tcat} or \cite[Chapters 1-3]{BaKi} as well an introduction to the Witt group of nondegenerate braided fusion categories $\mathcal{W}$.  Sections \ref{sec:classical} and \ref{sec:moddata} then give an introduction to the modular tensor categories $\mathcal{C}(\mathfrak{sl}_3,k)$, the main objects of our study, followed by detailed descriptions through the end of the section of the aforementioned simple, completely anisotropic building blocks that will be used for the classification.  Section \ref{sec:modcon} discusses the classification of symmetric $\mathfrak{sl}_3$ modular invariants due to Gannon \cite{gannon} and all relations coming from conformal embeddings of vertex operator algebras, leading up to the classification of Witt group relations in Section \ref{sec:class}.  We end with a brief discussion on how the methods presented in this paper may be generalized to a classification of all Witt group relations coming from $\mathcal{C}(\mathfrak{g},k)$ and the major obstructions to this goal.


\section{Preliminaries}\label{sec:prelims}

\subsection{Nondegenerate braided fusion categories}\label{firstone}

\par We assume familiarity with the basic definitions and results found for example in \cite{tcat}, but will give a brief recollection at this point.  In the remainder of this section $\mathbbm{k}$ will be an algebraically closed field of characteristic zero.

\par A \emph{fusion category} over $\mathbbm{k}$ is a $\mathbbm{k}$-linear semisimple rigid tensor category with finitely many simple objects, finite dimensional spaces of morphisms, and a simple unit object $\1$.  For brevity the set of simple objects of a fusion category $\mathcal{C}$ will be denoted $\mathcal{O}(\mathcal{C})$.  We will identify the unique (up to tensor equivalence) fusion category with one simple object with $\Vep$, the category of finite dimensional vector spaces over $\mathbbm{k}$.  Given two braided fusion categories $\mathcal{C}$ and $\mathcal{D}$, \emph{Deligne's tensor product} $\mathcal{C}\boxtimes\mathcal{D}$ is a new braided fusion category which can be realized as the completion of the $\mathbbm{k}$-linear direct product $\mathcal{C}\otimes_\mathbbm{k}\mathcal{D}$ under direct sums and subobjects under our current assumptions \cite[Section 4.6]{tcat}.

\par A set of natural isomorphisms
\begin{equation}\label{braiding}
c_{X,Y}:X\otimes Y\stackrel{\sim}{\longrightarrow}Y\otimes X
\end{equation}
satisfying compatibility relations \cite[Section 8.1]{tcat} for all $X,Y$ in a fusion category $\mathcal{C}$ is called a \emph{braiding} on $\mathcal{C}$ and we will therefore refer to $\mathcal{C}$ as a \emph{braided fusion category}.  There is an alternative \emph{reverse braiding} for any braided category given by $\tilde{c}_{X,Y}=c_{Y,X}^{-1}$ and the resulting braided category is denoted $\mathcal{C}^\text{rev}$.

\begin{example}[Pointed fusion categories]\label{pointed}  Special distinction goes to fusion categories $\mathcal{C}$ in which every object $X\in\mathcal{O}(\mathcal{C})$ is \emph{invertible}, i.e.\ the \emph{evaluation} $\text{ev}_X:X^\ast\otimes X\longrightarrow \1$ and \emph{coevaluation} $\text{coev}_X:\1\longrightarrow X\otimes X^\ast$ maps coming from the rigidity of $\mathcal{C}$ are isomorphisms.  Categories in which every $X\in\mathcal{O}(\mathcal{C})$ is invertible are called \emph{pointed}, while the maximal pointed subcategory of a braided fusion category $\mathcal{C}$ will be denoted $\mathcal{C}_\text{pt}$.  Pointed braided fusion categories were classified by Joyal and Street in \cite[Section 3]{joyalstreet} (see also \cite[Section 8.4]{tcat}).  If a pointed fusion category is braided, due to (\ref{braiding}) the set of simple objects forms a finite abelian group under the tensor product, which we will call $A$.   Recall that a quadratic form on $A$ with values in $\mathbbm{k}^\times$ is a function $q:A\to\mathbbm{k}^\times$ such that $q(-x)=q(x)$ and $b(x,y)=q(x+y)/(q(x)q(y))$ is bilinear for all $x,y\in A$.  To each pair $(A,q)$ there exists a braided fusion category $\mathcal{C}(A,q)$ that is unique up to braided equivalence whose simple objects are labelled by the elements of $A$.
\end{example}

\par One might identify symmetric braidings (those for which $c_{Y,X}\circ c_{X,Y}=\text{id}_{X\otimes Y}$) as the most elementary of braidings as Deligne \cite{deligne1}\cite{deligne2}\cite[Section 9.9]{tcat} proved that all symmetric fusion categories must come from the representation theory of finite groups.  In the spirit of gauging how far a braiding is from being symmetric, if $c_{Y,X}\circ c_{X,Y}=\text{id}_{X\otimes Y}$ for any objects $X,Y\in\mathcal{C}$, we say $X$ and $Y$ \emph{centralize} one another \cite[Section 2.2]{mug1}.

\begin{defi}\label{def:centralizer}
If $\mathcal{D}$ is a subcategory of braided fusion category $\mathcal{C}$ that is closed under tensor products then $\mathcal{D}'\subset\mathcal{C}$ the \emph{centralizer of $\mathcal{D}$ in $\mathcal{C}$} is the full subcategory of objects of $\mathcal{C}$ that centralize each object of $\mathcal{D}$. A braided fusion category is known as \emph{nondegenerate} if $\mathcal{C}'\simeq\Vep$.
\end{defi}

\begin{note}One can think of nondegenerate braided fusion categories as those which are furthest from symmetric as possible.\end{note}

\begin{example}[Metric groups]\label{metric}
If the symmetric bilinear form $b(-,-)$ associated with a pair $(A,q)$ (as in Example \ref{pointed}) is nondegenerate, then the pair $(A,q)$ is called a \emph{metric group}.  It is known \cite[Example 8.13.5]{tcat} that the category $\mathcal{C}(A,q)$ is nondegenerate if and only if $(A,q)$ is a metric group.  For instance let $A:=\mathbb{Z}/3\mathbb{Z}$ (considered as the set $\{0,1,2\}$ with the operation of addition modulo 3).  The following functions are quadratic forms on $A$ with values in $\mathbb{C}^\times$:
\begin{align*}
q_{\omega}:&\,A\longrightarrow\mathbb{C}^\times& q_{\omega^2}:&\,A\longrightarrow\mathbb{C}^\times \\
&x\mapsto(\omega)^{x^2}& &x\mapsto(\omega^2)^{x^2} \\
\end{align*}
where $\omega=\exp(2\pi i/3)$.  These quadratic forms equip $\mathcal{C}(\mathbb{Z}/3\mathbb{Z},q_\omega)$ and $\mathcal{C}(\mathbb{Z}/3\mathbb{Z},q_{\omega^2})$ with the structure of nondegenerate braided fusion categories which are not braided equivalent.
\end{example}


\subsection{Fusion subcategories and prime decomposition}\label{subcategories}

\par The assumptions required of a fusion subcategory are very few in number.

\begin{defi}
A full subcategory $\mathcal{D}$ of a fusion category $\mathcal{C}$ is a \emph{fusion subcategory} if $\mathcal{D}$ is closed under tensor products.
\end{defi}

\par It would not be unreasonable to assume that rigidity and existence of the unit object of $\mathcal{D}$ be required in the definition above, but both are consequences of closure under tensor products.  Specifically Lemma 4.11.3 of \cite{tcat} gives that for each simple object $X$ there exists $n\in\mathbb{Z}_{>0}$ such that $\Hom(\1,X^n)\neq0$.  And by adjointness of duality \cite[Proposition 2.10.8]{tcat} $\Hom(X^\ast,X^{n-1})\neq0$ as well.  Thus $\1,X^\ast\in\mathcal{C}$ are direct summands of sufficiently large powers of $X$.
\begin{defi}\label{def:simple}
A fusion category with no proper, nontrivial fusion subcategories is called \emph{simple}, while a nondegenerate fusion category with no proper, nontrivial, nondegenerate fusion subcategories is called \emph{prime}.
\end{defi}
The existence of a decomposition of a nondegenerate braided fusion category into a product of prime fusion subcategories was given by M\"uger \cite[Section 4.1]{mug1} under limited assumptions and proved in the following generality in Theorem 3.13 of \cite{DGNO}.

\begin{mytheorem}\label{primedecomp}
Let $\mathcal{C}\neq\Vep$ be a nondegenerate braided fusion category.  Then
\[\mathcal{C}\simeq\mathcal{C}_1\boxtimes\cdots\boxtimes\mathcal{C}_n,\]
where $\mathcal{C}_1,\ldots,\mathcal{C}_n$ are prime nondegenerate subcategories of $\mathcal{C}$.
\end{mytheorem}

\par To construct such a decomposition one can identify a nontrivial nondegenerate braided fusion subcategory $\mathcal{D}$ inside of a given nondegenerate braided fusion category $\mathcal{C}$ and by Theorem 4.2 of \cite{mug1}, $\mathcal{C}\simeq\mathcal{D}\boxtimes\mathcal{D}'$ is a braided equivalence.  In future sections this process will be referred to as \emph{M\"uger's decomposition}.


\subsection{Modular Categories}\label{modular}

\par Recall the natural isomorphisms $a_V:V\stackrel{\sim}{\longrightarrow} V^{\ast\ast}$ for any finite dimensional vector space $V$ over $\mathbbm{k}$ from elementary linear algebra.  This collection of natural isomorphisms is a \emph{pivotal structure} on $\Vep$, i.e. they satisfy $a_{V\otimes W}=a_V\otimes a_W$ for any finite dimensional vector spaces $V$ and $W$.  A pivotal structure on a general tensor category $\mathcal{C}$ allows us to define a categorical analog of \emph{trace}, $\Tr(f)\in\mathbbm{k}$ for any morphism $f:X\longrightarrow X^{\ast\ast}$ \cite[Section 4.7]{tcat} given by
\[\Tr(f):\1\stackrel{\text{coev}_X}{\longrightarrow}X\otimes X^\ast\stackrel{f\otimes\text{id}_{X^\ast}}{\longrightarrow}X^{\ast\ast}\otimes X^\ast\stackrel{\text{eval}_{X^\ast}}{\longrightarrow}\1.\]
Tensor categories with a pivotal structure $a_X:X\stackrel{\sim}{\longrightarrow}X^{\ast\ast}$ for all objects $X$ will be called \emph{pivotal} themselves.

\begin{defi}\label{spherical}
The (categorical or quantum) \emph{dimension} of an object $X$ in pivotal tensor category $\mathcal{C}$ is $\dim(X):=\Tr(a_X)\in\mathbbm{k}$ while
\[\dim(\mathcal{C}):=\sum_{X\in\mathcal{O}(\mathcal{C})}|X|^2\]
where $|X|^2=\Tr(a_X)\Tr((a_X^{-1})^\ast)$.  A pivotal structure on a tensor category is called \emph{spherical} if $\dim(X)=\dim(X^\ast)$ for all $X\in\mathcal{O}(\mathcal{C})$, while spherical braided fusion categories are called \emph{pre-modular}.
\end{defi}

\begin{example}[Vector spaces]
\par There is only one simple object in $\Vep$ up to isomorphism, the one-dimensional $\mathbbm{k}$-vector space $\1$, and the aforementioned pivotal structure given by $a_V:V\stackrel{\sim}{\longrightarrow} V^{\ast\ast}$ by $v\mapsto\{f\mapsto f(v)\}$ is spherical.  It is easily verified that $\dim(\1)=1$ and since the categorical notion of dimension is additive, then in this case the categorical dimension matches the usual notion of the dimension of a $\mathbbm{k}$-vector space.
\par More generally since the categories $\mathcal{C}(A,q)$ are pointed, the evaluation, coevaluation, and spherical structure can be realized by identity maps.  So $\dim(X)=1$ for all simple $X\in\mathcal{C}(A,q)$ and all metric groups $(A,q)$ and one can conclude $\dim(\mathcal{C}(A,q))=|A|$.
\end{example}

\par There is a second notion of dimension defined in terms of the Grothendieck ring $K(\mathcal{C})$ of a fusion category $\mathcal{C}$.  As noted in Section 3.3 of \cite{tcat}, there exists a unique ring homomorphism $\FP:K(\mathcal{C})\longrightarrow\mathbbm{R}$ such that $\FP(X)>0$ for any $0\neq X\in\mathcal{C}$. This \emph{Frobenius-Perron dimension} gives an analog to the dimension of the category $\mathcal{C}$ itself as in Definition \ref{spherical}, given by
\[\FP(\mathcal{C})=\sum_{X\in\mathcal{O}(\mathcal{C})}\FP(X)^2.\]

\par Fusion categories for which $\FP(\mathcal{C})=\dim(\mathcal{C})$ are called \emph{pseudo-unitary} and it is known that for such a category there exists a unique spherical structure with $\FP(X)=\dim(X)$ for all $X\in\mathcal{O}(\mathcal{C})$, allowing us to only consider $\dim(X)$ in these cases.  It will be important to future computations that $\dim(X)>0$ for pseudo-unitary fusion categories.

\par If a braided fusion category is equipped with a spherical structure, there exist natural isomorphisms $\theta_X:X\stackrel{\sim}{\longrightarrow}X$ for all $X\in\mathcal{C}$ known as \emph{twists} (or the ribbon structure) compatible with the braiding isomorphisms found in (\ref{braiding}) of Section \ref{firstone} \cite[Section 8.10]{tcat}.  In the case of pointed fusion categories $\mathcal{C}(A,q)$ (Example \ref{pointed}), for any $x\in A$ the map $\theta_x=b(x,x)\text{id}_x$ defines a ribbon structure.  The diagonal matrix consisting of the twists $\theta_X$ over all $X\in\mathcal{O}(\mathcal{C})$ is called the $T$-matrix of $\mathcal{C}$.

\par Finally we end this subsection by tying the notions of trace and dimension in spherical categories to the nondegeneracy conditions defined by the centralizer construction (Definition \ref{def:centralizer}).

\begin{defi}\label{def:modular}
The $S$-matrix of a pre-modular category $\mathcal{C}$ is the matrix $(s_{XY})_{X,Y\in\mathcal{O}(\mathcal{C})}$ where $s_{X,Y}:=\Tr(c_{Y,X}\circ c_{X,Y})$.  A pre-modular category is \emph{modular} if the determinant of its $S$-matrix is nonzero.
\end{defi}
\begin{note}
It is well-known that a pre-modular category $\mathcal{C}$ is modular if and only if it is nondegenerate ($\mathcal{C}'=\Vep$). \cite[Proposition 3.7]{DGNO}\cite{mug1}
\end{note}


\subsection{\'Etale Algebras}\label{etale}

\par For this exposition, an \emph{algebra} $A$ in a fusion category $\mathcal{C}$ is an associative algebra with unit which is equipped with a multiplication map $m:A\otimes A\longrightarrow A$.  If $m$ splits as a morphism of $A$-bimodules, we refer to $A$ as \emph{separable}.  This criterion ensures that $\mathcal{C}_A$, the category of right $A$-modules is semisimple, and also $_A\,\mathcal{C}$, $_A\,\mathcal{C}_A$, the categories of left $A$-modules and $A$-bimodules respectively \cite[Proposition 2.7]{DMNO}.

\begin{defi}
An algebra $A$ in a fusion category $\mathcal{C}$ is \emph{\'etale} if it is both commutative and separable.  This algebra is \emph{connected} if $\dim_\mathbbm{k}\Hom(\1,A)=1$.
\end{defi}
\begin{note}
\'Etale algebras have also been referred to as \emph{condensable} algebras in the physics literature.  The following description of the categories $\mathcal{C}_A$ when $A$ is connected \'etale is summarized from Sections 3.3 and 3.5 of \cite{DMNO}.
\end{note}

\par Braidings on $\mathcal{C}$ give rise to functors $G:\tensor[]{\mathcal{C}}{_A}\longrightarrow \tensor[_A]{\mathcal{C}}{_A}$ defined as $M\mapsto M_-$ (the identity map as right $A$-modules), where the left $A$-module structure on $M_-$ is given as composition of the reverse braiding with the right $A$-module structure map $\rho$:
\[A\otimes M\stackrel{\tilde{c}_{M,A}}{\longrightarrow}M\otimes A\stackrel{\rho}{\longrightarrow}M.\]
The commutativity of $A$ implies $\tensor[_A]{\mathcal{C}}{_A}$ is a tensor category and thus the above functor $G$ provides a tensor structure for $\tensor[]{\mathcal{C}}{_A}$ which we denote $\otimes_A$.  One can also define a tensor structure, opposite to the one above, on $\tensor[]{\mathcal{C}}{_A}$ by composing the right $A$-module structure map with the usual braiding.  We will denote the resulting left $A$-module produced from $M$ as $M_+$.

\par With the tensor structure defined on $\tensor[]{\mathcal{C}}{_A}$ by the functor $G$, the free module functor $F:\mathcal{C}\to\tensor[]{\mathcal{C}}{_A}$ is a tensor functor.  In particular $F(\1)=A$ is the unit object of $\tensor[]{\mathcal{C}}{_A}$ which is simple by the assumption that $A$ is connected.  The category $\tensor[]{\mathcal{C}}{_A}$ is also rigid since any object $M\in\tensor[]{\mathcal{C}}{_A}$ is a direct summand of the rigid object
\[F(M)=M\otimes A=M\otimes_A(A\otimes A).\]
The above discussion implies $\tensor[]{\mathcal{C}}{_A}$ is a fusion category when $A\in\mathcal{C}$ is connected \'etale.  Unfortunately the category $\tensor[]{\mathcal{C}}{_A}$ is \emph{not} braided in general.  The issue lies in the inherent \emph{choice} of a left $A$-module structure on a given right $A$-module $M\in\tensor[]{\mathcal{C}}{_A}$.

\begin{defi}
If $\text{id}_M:M_-\longrightarrow M_+$ is an isomorphism of $A$-bimodules for $M\in\tensor[]{\mathcal{C}}{_A}$, we say that $M$ is \emph{dyslectic} (also called \emph{local} in the literature).
\end{defi}

Pareigis \cite{pareigis} originally studied the full subcategory of $\tensor[]{\mathcal{C}}{_A}$ consisting of dyslectic $A$-modules, denoted by $\mathcal{C}_A^0$ which is the correct subcategory of $\tensor[]{\mathcal{C}}{_A}$ to study to ensure a braiding exists (see also \cite{KiO}).  That is if $\mathcal{C}$ is a braided fusion category and $A\in\mathcal{C}$ a connected \'etale algebra, then $\mathcal{C}_A^0$ is a braided fusion category and furthermore if $\mathcal{C}$ is nondegenerate then $\mathcal{C}_A^0$ is nondegenerate as well.
\begin{defi}\label{def:anisotropic}
A braided fusion category $\mathcal{C}$ is \emph{completely anisotropic} if the only connected \'etale algebra in $\mathcal{C}$ is the unit object $\1$.
\end{defi}


\subsection{The Witt group of nondegenerate braided fusion categories}\label{witt}

\par Tensor categories are often regarded as a categorical analog of rings and there is a categorical construction which (in some ways) mimics the center of a ring.  The \emph{Drinfeld center} of a monoidal (tensor, fusion) category $\mathcal{C}$ is the category whose objects are pairs $(X,\{\gamma_{X,Y}\}_{Y\in\mathcal{C}})$ consisting of an object of $X\in\mathcal{C}$ and natural isomorphisms
\[\gamma_{X,Y}:X\otimes Y\stackrel{\sim}{\longrightarrow}Y\otimes X\]
for all objects of $Y\in\mathcal{C}$ that satisfy the same compatibility conditions as braidings found in (1) of Section 2.1; i.e.\ this definition is imposed so that $Z(\mathcal{C})$ is naturally braided.  Where the analogy to the center of a ring falls apart is that in general $Z(\mathcal{C})$ is much \emph{larger} than $\mathcal{C}$ as the same object $X\in\mathcal{C}$ may have many distinct braidings $\gamma_{X,Y}$ to be paired with it.  If $\mathcal{C}$ is a braided fusion category, the functors $\mathcal{C},\mathcal{C}^\text{rev}\longrightarrow Z(\mathcal{C})$ mapping objects $X$ to themselves paired with their inherent braiding isomorphisms in $\mathcal{C},\mathcal{C}^\text{rev}$ are fully faithful and their images centralize one another, giving a braided tensor functor
\begin{equation}
\mathcal{C}\boxtimes\mathcal{C}^\text{rev}\longrightarrow Z(\mathcal{C})\label{centeriso}
\end{equation}
which has been shown to be an isomorphism if and only if $\mathcal{C}$ is modular \cite[Proposition 3.7]{DGNO}\cite[Theorem 7.10]{mug2}.

\par It is not obvious whether a given nondegenerate braided fusion category arises as the Drinfeld center of another.  The Witt group of nondegenerate braided fusion categories can be seen as a device for organizing nondegenerate braided fusion categories by equating those that differ only by the Drinfeld center of another.

\begin{defi}
The \emph{Witt group of nondegenerate braided fusion categories} (hereby called \emph{the Witt group}, or simply $\mathcal{W}$) is the set of equivalence classes of nondegenerate braided fusion categories $[\mathcal{C}]$ where $[\mathcal{C}]=[\mathcal{D}]$ if there exist fusion categories $\mathcal{A}_1$ and $\mathcal{A}_2$ such that $\mathcal{C}\boxtimes Z(\mathcal{A}_1)\simeq\mathcal{D}\boxtimes Z(\mathcal{A}_2)$ as braided fusion categories.
\end{defi}
The title \emph{group} is justified as the Deligne tensor product equips $\mathcal{W}$ with a commutative monoidal structure (with unit $[\Vep]$) while (\ref{centeriso}) implies that $[\mathcal{C}]^{-1}=[\mathcal{C}^\text{rev}]$ \cite[Lemma 5.3]{DMNO}.

\par Completely anisotropic categories (Definition \ref{def:anisotropic}) play a special role in the study of $\mathcal{W}$.  As noted in Theorem 5.13 of \cite{DMNO} each Witt equivalence class in $\mathcal{W}$ contains a completely anisotropic category that is unique up to braided equivalence.  To produce such a representative one can locate a maximal connected \'etale algebra $A\in\mathcal{C}$ and the passage to the category of dyslectic $A$-modules $\mathcal{C}_A^0$ does not change the Witt equivalency class, i.e.\  $[\mathcal{C}_A^0]=[\mathcal{C}]$ \cite[Proposition 5.4]{DMNO}.

\par One impetus to understanding the structure of $\mathcal{W}$ is that the decomposition of a nondegenerate braided fusion category given in Proposition \ref{primedecomp} is not unique in general.  The extent of this lack of of uniqueness is illustrated in Section 4.2 of \cite{mug1}.


\par The last tool we will set up in this section is a numerical invariant that will allow us to quickly prove that Witt equivalence classes of categories are distinct.  Assume for the rest of this section that $\mathcal{C}$ is a modular tensor category over $\mathbb{C}$ (Definition \ref{def:modular}).

\par Recall the \emph{multiplicative central charge} $\xi(\mathcal{C})\in\mathbb{C}$ \cite[Section 8.15]{tcat} which satisfies the following important properties.
\begin{mylemma}\label{charge}  For any modular tensor categories $\mathcal{C}_1$ and $\mathcal{C}_2$
\begin{enumerate}
\item[\textnormal{(a)}]$\xi(\mathcal{C})$ is a root of unity,
\item[\textnormal{(b)}]$\xi(\mathcal{C}_1\boxtimes\mathcal{C}_2)=\xi(\mathcal{C}_1)\xi(\mathcal{C}_2)$, and
\item[\textnormal{(c)}]$\xi(\mathcal{C}^\text{rev})=\xi(\mathcal{C})^{-1}$.
\end{enumerate}
\end{mylemma}
The equivalence in (\ref{centeriso}) along with Lemma \ref{charge} (b),(c) implies that $\xi(Z(\mathcal{C}))=1$.  Lemma 5.27 of \cite{DMNO} proves further that multiplicative central charge is a numerical invariant of Witt equivalency classes.  This allow us to distinguish Witt equivalency classes and predict the possible order of elements in $\mathcal{W}$.


\section{The categories $\mathcal{C}(\mathfrak{sl}_3,k)$, $k\in\mathbb{Z}_{>0}$}

\subsection{Classical and Modular $\mathfrak{sl}_3$ theories}\label{sec:classical}

\par All notation and concepts in the classical theory are standard and can be found for instance in \cite{hump}, while a brief introduction to the modular theory can be found in Chapters 3.3 and 7 of \cite{BaKi}.

\par The Lie algebra $\mathfrak{sl}_3$ contains a two-dimensional Cartan subalgebra $\mathfrak{h}$ so it has two associated simple roots $\alpha_1,\alpha_2$ which span $\mathfrak{h}^\ast$ as a $\mathbb{C}$-vector space.  There is a nondegenerate symmetric bilinear form $\langle\cdot\,,\cdot\rangle$ on $\mathfrak{h}^\ast$ (normalized so that $\langle\alpha_i,\alpha_i\rangle=2$ for $i=1,2$) which allows us to consider $\varepsilon_1,\varepsilon_2$, the dual basis for $\mathfrak{h}^\ast$ with respect to this form.  The pair $(\mathfrak{h}^\ast,\langle\cdot\,,\cdot\rangle)$ is then a Euclidean space which we denote by $E$.  Finite-dimensional irreducible $\mathcal{U}(\mathfrak{sl}_3)$-modules are in one-to-one correspondence with pairs of nonnegative integers $(m_1,m_2)$.  Under this correspondence the pair $(m_1,m_2)$ is associated with the irreducible $\mathcal{U}(\mathfrak{sl}_3)$-module of highest weight $m_1\varepsilon_1+m_2\varepsilon_2$.  All such weights are called \emph{dominant} and their collection is denoted $\Lambda^+$.  Alternatively one can describe dominant weights as those that lie in the \emph{closure of the fundamental Weyl chamber}:
\[C:=\{x\in E:\langle x,\alpha_i\rangle\geq0\text{ for }i=1,2\}.\]
This classical setup is illustrated below in Figure \ref{fig:classical} where we denote the root corresponding to the adjoint representation of $\mathfrak{sl}_3$ by $\theta$ for future reference.

\begin{figure}[h!]
\centering
\begin{tikzpicture}[scale=1.5]
\fill [blue,opacity=.1] (0,0) -- (0,4*1.732/3) -- (2,4*1.732/3) -- (2,2*1.732/3) -- cycle;

\node at (0,0) {$\bullet$};
\node at (0,1.732/3) {$\bullet$};
\node at (0,1.732/3+0.195) {$\varepsilon_2$};
\node at (0,2*1.732/3) {$\bullet$};
\node at (0,3*1.732/3) {$\bullet$};
\node at (0,4*1.732/3) {$\bullet$};
\node at (1/2,0.5*1.732/3) {$\bullet$};
\node at (1/2+0.21,0.5*1.732/3+0.12) {$\varepsilon_1$};
\node at (1/2,1*1.732/3+0.5*1.732/3) {$\bullet$};
\node at (1/2+0.12,1*1.732/3+0.5*1.732/3+0.195) {$\theta$};
\node at (1/2,2*1.732/3+0.5*1.732/3) {$\bullet$};
\node at (1/2,3*1.732/3+0.5*1.732/3) {$\bullet$};
\node at (1,1*1.732/3) {$\bullet$};
\node at (1,2*1.732/3) {$\bullet$};
\node at (1,3*1.732/3) {$\bullet$};
\node at (1,4*1.732/3) {$\bullet$};
\node at (1.5,1*1.732/3+0.5*1.732/3) {$\bullet$};
\node at (1.5,2*1.732/3+0.5*1.732/3) {$\bullet$};
\node at (1.5,3*1.732/3+0.5*1.732/3) {$\bullet$};
\node at (2,2*1.732/3) {$\bullet$};
\node at (2,3*1.732/3) {$\bullet$};
\node at (2,4*1.732/3) {$\bullet$};

\foreach \i in {-1,...,1}
\foreach \j in {-1,...,3}
{
\draw[fill=black] (\i+1/2,\j*1.732/3+1.732/6) circle (0.01);
}

\foreach \i in {-1,...,2}
\foreach \j in {-1,...,4}
{
\draw[fill=black] (\i,\j*1.732/3) circle (0.01);
}

\draw[->,dotted] (0,0) to (1,0) node[below right] {$\alpha_1$};
\draw[->,dotted] (0,0) to (1/2,1.732/2);
\draw[->,dotted] (0,0) to (-1/2,1.732/2) node[above left] {$\alpha_2$};

\draw[->,dashed,thick] (0,0) to (2,2*1.732/3);
\draw[->,dashed,thick] (0,0) to (0,4*1.732/3);

\end{tikzpicture}
    \caption{Classical $\mathfrak{sl}_3$ theory}%
    \label{fig:classical}%
\end{figure}
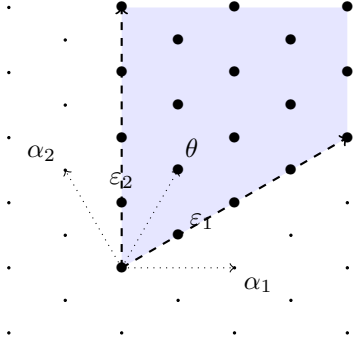

\par If $\mathfrak{g}$ is a finite-dimensional simple Lie algebra and $\hat{\mathfrak{g}}$ is the corresponding affine Lie algebra, then for all $k\in\mathbb{Z}_{>0}$ one can associate a pseudo-unitary modular tensor category $\mathcal{C}(\mathfrak{g},k)$ consisting of highest weight integrable $\hat{\mathfrak{g}}$-modules of level $k$.  This category was studied by Andersen and Paradowski \cite{ap} and Finkelberg \cite{fink} later proved that $\mathcal{C}(\mathfrak{g},k)$ is equivalent to the semisimple portion of the representation category of Lusztig's quantum group $\mathcal{U}_q(\mathfrak{g})$ when $q=e^{\pi i/(k+h^\vee)}$ where $h^\vee$ is the dual coxeter number for $\mathfrak{g}$ \cite[Chapter 7]{BaKi}.  Simple objects of $\mathcal{C}(\mathfrak{sl}_3,k)$ are in one-to-one correspondence with pairs of nonnegative integers $(m_1,m_2)$ such that $m_1+m_2\leq k$.

\par There is a geometric description of these pairs as in the classical setting.   The \emph{Weyl alcove} is defined as $C_0:=\{x\in E:\langle x,\theta\rangle\leq k\}$ where $\theta$ is again the highest weight corresponding to the adjoint representation of $\mathfrak{sl}_3$.  Simple objects in $\mathcal{C}(\mathfrak{sl}_3,k)$ then correspond to the weights in the intersection $\Lambda_0:=\Lambda^+\cap C_0$.  Alternatively, $C_0$ can be defined in terms of the \emph{quantum Weyl group} $\mathfrak{W}_0$ generated by reflections $\tau_1,\tau_2$, the reflections through the hyperplanes $T_1,T_2:=\{x\in E:\langle x+\rho,\alpha_i\rangle=0\}$ where $\rho:=(1/2)(\alpha_1+\alpha_2)$, and $\tau_3$, the reflection through the hyperplane $T_3:=\{x\in E:\langle x,\theta\rangle=k+1\}$.  These concepts are illustrated in Figure \ref{fig:modular} below.

\par Very often we will use $\lambda$ or $(m_1,m_2)$ to represent the simple object of $\mathcal{C}(\mathfrak{sl}_3,k)$ corresponding to the weight $\lambda$ or $m_1\varepsilon_1+m_2\varepsilon_2$ in $\Lambda_0$.  Sentences such as ``$\lambda\otimes\mu$ contains $\nu$ as a direct summand'' or reference to ``$\dim(2,3)$'' will then make sense in future contexts.
\begin{figure}[h!]
\centering
\begin{tikzpicture}[scale=1.5]
\fill [blue,opacity=.1] (0,0) -- (0,2*1.732/2) -- (3*1/2,1*1.732/2) -- cycle;
\node at (0,1.732/3+0.195) {$\varepsilon_2$};
\node at (1/2+0.21,0.5*1.732/3+0.12) {$\varepsilon_1$};
\node at (0,0) {$\bullet$};
\node at (0,1.732/3) {$\bullet$};
\node at (0,2*1.732/3) {$\bullet$};
\node at (0,3*1.732/3) {$\bullet$};
\node at (1/2,0.5*1.732/3) {$\bullet$};
\node at (1/2,1*1.732/3+0.5*1.732/3) {$\bullet$};
\node at (1/2+0.12,1*1.732/3+0.5*1.732/3+0.195) {$\theta$};
\node at (1/2,2*1.732/3+0.5*1.732/3) {$\bullet$};
\node at (1,1*1.732/3) {$\bullet$};
\node at (1,2*1.732/3) {$\bullet$};
\node at (1.5,1*1.732/3+0.5*1.732/3) {$\bullet$};

\foreach \i in {-2,...,2}
\foreach \j in {-2,...,4}
{
\draw[fill=black] (\i+1/2,\j*1.732/3+1.732/6) circle (0.01);
}

\foreach \i in {-1,...,3}
\foreach \j in {-2,...,5}
{
\draw[fill=black] (\i,\j*1.732/3) circle (0.01);
}

\draw[->,dotted] (0,0) to (1,0) node[below right] {$\alpha_1$};
\draw[->,dotted] (0,0) to (1/2,1.732/2);
\draw[->,dotted] (0,0) to (-1/2,1.732/2) node[above left] {$\alpha_2$};

\draw[<->,dashed,thick] (-1.5,-2*1.732/3) to node[circle,draw=black,pos = 0.6,inner sep=1pt,fill=white,solid,thin] {$T_2$} (3,2.5*1.732/3);
\draw[<->,dashed,thick] (-1/4,-2*1.732/3) to node[circle,draw=black,pos = 0.65,inner sep=1pt,fill=white,solid,thin] {$T_1$} (-1/4,5*1.732/3);
\draw[<->,dashed,thick] (-1,5*1.732/3) to node[circle,draw=black,pos = 0.5,inner sep=1pt,fill=white,solid,thin] {$T_3$} (3,1.732/3);

\end{tikzpicture}
    \caption{Modular $\mathfrak{sl}_3$ theory at level $k=3$}%
    \label{fig:modular}%
\end{figure}
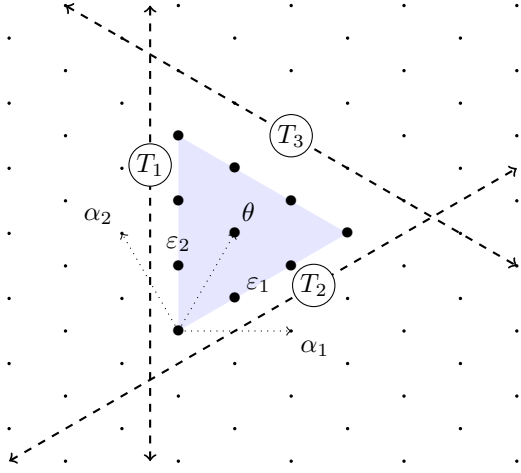


\subsection{Numerical Data for $\mathcal{C}(\mathfrak{sl}_3,k)$}\label{sec:moddata}
Explicit formulas for the modular data of $\mathcal{C}(\mathfrak{g},k)$ are well-known \cite[Section 3.3]{BaKi} but we include these computations in the case of $\mathfrak{sl}_3$ for the reader's convenience.  For example we have the following formula for twists on simple objects.
\begin{mylemma}\label{twist}
For all $(m_1,m_2)\in\Lambda_0$
\[\theta_{m_1,m_2}=\exp\left(\frac{m_1^2+3m_1+m_1m_2+3m_2+m_2^2}{3(k+3)}\cdot2\pi i\right).\]
\end{mylemma}
\begin{proof}
In Lusztig's quantum group realization of $\mathcal{C}(\mathfrak{sl}_3,k)$, $q=\exp\left(\pi i/(k+3)\right)$ and one can compute twists using the formula $\theta_\lambda=q^{\langle\lambda,\lambda+2\rho\rangle}$ for any $\lambda\in\Lambda_0$ by Theorem 3.3.20 of \cite{BaKi}.
\end{proof}
To simplify notation in the formula for the (categorical) dimensions of simple objects of $\mathcal{C}(\mathfrak{sl}_3,k)$ we define the quantum integers
\[[n]:=\dfrac{q^n-q^{-n}}{q-q^{-1}}=q^{n-1}+q^{n-3}+\cdots+q^{-(n-3)}+q^{-(n-1)}.\]
This notation produces a formula for dimensions akin to the classical Weyl dimension formula for the dimensions of finite-dimensional $\mathcal{U}(\mathfrak{sl}_3)$-modules as complex vector spaces \cite[Equation (3.3.5)]{BaKi}.
\begin{mylemma}[Quantum Weyl dimension formula]
For all $(m_1,m_2)\in\Lambda_0$
\begin{align*}\dim(m_1,m_2)&=\frac{\left[\langle\alpha_1,m_1\varepsilon_1+m_2\varepsilon_2+\rho\rangle\right]\left[\langle\alpha_2,m_1\varepsilon_1+m_2\varepsilon_2+\rho\rangle\right]}{[\langle\alpha_1,\rho\rangle][\langle\alpha_2,\rho\rangle]} \\
&=\dfrac{1}{[2]}[m_1+1][m_2+1][m_1+m_2+2].
\end{align*}
\end{mylemma}
Using elementary trigonometry, these dimensions can be expressed solely in terms of sines. Since the argument of $q^n$ is $n\pi/(k+3)$ as illustrated in Figure \ref{fig:modulus}, then
\[\left|[n]\right|=\left|\dfrac{q^n-q^{-n}}{q-q^{-1}}\right|=\dfrac{\sin\left(\dfrac{n\pi}{k+3}\right)}{\sin\left(\dfrac{\pi}{k+3}\right)}.\]
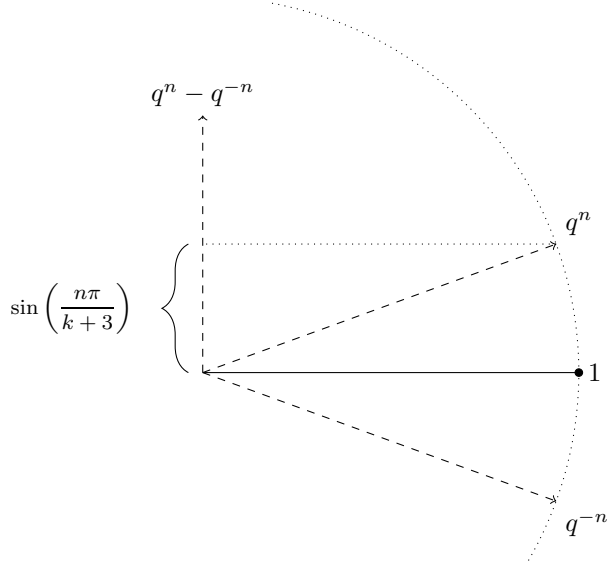
\begin{figure}[h!]
\centering
\begin{tikzpicture}
\draw[fill=black] (5,0) circle (0.05);
\draw (0,0) to (5,0) node[right] {1};
\draw[dashed,->] (0,0) to (-20:5) node[below right] {$q^{-n}$};
\draw[dashed,->] (0,0) to (20:5) node[above right] {$q^n$};
\draw[dashed,->] (0,0) to (0,5*0.68404) node[above] {$q^n-q^{-n}$};
\draw[dotted] (20:5) to (0,5/2*0.68404);
\draw[dotted] (-30:5) arc (-30:80:5);
\draw [decorate,decoration={brace,amplitude=10pt},xshift=-4pt,yshift=0pt]
(-0.05,0) -- (-0.05,5/2*0.68404) node [black,midway,xshift=-1.55cm] 
{\footnotesize $\sin\left(\dfrac{n\pi}{k+3}\right)$};
\end{tikzpicture}
    \caption{Modulus of $q^n-q^{-n}$}%
    \label{fig:modulus}%
\end{figure}
\par But since $\mathcal{C}(\mathfrak{sl}_3,k)$ is pseudo-unitary (Section \ref{modular}), then $|\dim(m_1,m_2)|=\dim(m_1,m_2)$ and the above reasoning along with the quantum Weyl formula imply the following proposition.
\begin{mytheorem}\label{prop:dim}
For all $(m_1,m_2)\in\Lambda_0$
\[\dim(m_1,m_2)=\dfrac{\sin\left(\dfrac{(m_1+1)\pi}{k+3}\right)\sin\left(\dfrac{(m_2+1)\pi}{k+3}\right)\sin\left(\dfrac{(m_1+m_2+2)\pi}{k+3}\right)}{\sin\left(\dfrac{2\pi}{k+3}\right)\sin^2\left(\dfrac{\pi}{k+3}\right)}.\]
\end{mytheorem}
Lastly we recall a result influenced by Andersen and Paradowski and proven by Sawin as Corollary 8 in \cite{Sawin03}, giving a formula for the fusion rules in $\mathcal{C}(\mathfrak{sl}_3,k)$.
\begin{mytheorem}[Quantum Racah formula]\label{quantumracah} If $\lambda,\gamma,\eta\in\Lambda_0$ then $N_{\lambda,\gamma}^\eta:=\dim_\mathbb{C}\Hom(\eta,\lambda\otimes\gamma)$ is given by
\[N_{\lambda,\gamma}^\eta=\sum_{\tau\in\mathfrak{W}_0}(-1)^{\ell(\tau)}m_\gamma(\tau(\eta)-\lambda),\]
\noindent where $\ell(\tau)$ is the length of a reduced expression of $\tau\in\mathfrak{W}_0$ in terms of $\tau_1,\tau_2,\tau_3$ and $m_\lambda(\mu)$ is the dimension of the $\mu$-weight space in the classical representation of highest weight $\lambda$.
\end{mytheorem}

\par As in Lemma 1 of \cite{Sawin06} this formula can be used to identify particular direct summands of tensor products of simple objects in $\mathcal{C}(\mathfrak{sl}_3,k)$.  Based on slight notational discrepancies in the Quantum Racah formula in \cite{Sawin06}, we provide a proof here based on that of Sawin's.

\begin{mylemma}[Sawin]\label{linkage}For any $\sigma$ in the classical Weyl group $\mathfrak{W}$, and any $\gamma,\lambda\in\Lambda_0$, if $\lambda+\sigma(\gamma)\in\Lambda_0$, then $\lambda\otimes\gamma$ contains $\lambda+\sigma(\gamma)$ as a direct summand with multiplicity one.\end{mylemma}

\begin{proof}
\noindent Assume that $\lambda'\notin \Lambda_0$ is any weight conjugate to $\lambda\in \Lambda_0$ under the action of $\mathfrak{W}_0$.  Explicitly, there exists $(\tau_{i_1}\tau_{i_2}\cdots\tau_{i_t})\in\mathfrak{W}_0$ (written as a reduced expression in the generating simple reflections) such that
\begin{equation}(\tau_{i_1}\tau_{i_2}\cdots\tau_{i_t})(\lambda')=\lambda.\label{dagger}\end{equation}
Now let $\eta\in\Lambda_0$ be arbitrary.  The hyperplane of reflection corresponding to $\tau_{i_t}$ lies between $\lambda'$ and $\eta$ by assumption, so ${\|\tau_{i_t}(\lambda')-\eta\|<\|\lambda'-\eta\|}$.  Repeating this argument over all simple reflections in (\ref{dagger}) shows that
\begin{equation}\|\lambda-\eta\|<\|\lambda'-\eta\|.\label{dagger2}\end{equation}
\noindent With reference to the summands appearing in Proposition \ref{quantumracah}, assume that $m_\gamma(\tau(\lambda+\sigma(\gamma))-\lambda)\neq0$ for some non-trivial $\tau\in\mathfrak{W}_0$.  Then
\begin{equation}\|\tau(\lambda+\sigma(\gamma))-\lambda\|\leq\|\gamma\|\label{highest}\end{equation}
because $\gamma$ is heighest weight.  Since $\lambda+\sigma(\gamma)\in\Lambda_0$ and $\tau(\lambda+\sigma(\gamma))$ is not, (\ref{dagger2}) implies
\begin{equation*}\|\gamma\|=\|\sigma(\gamma)\|=\|(\lambda+\sigma(\gamma))-\lambda\|<\|\tau(\lambda+\sigma(\gamma))-\lambda\|\end{equation*}
contradicting the highest weight inequality in (\ref{highest}).  Thus $m_\gamma(\tau(\lambda+\sigma(\gamma))-\lambda)$ is possibly nonzero if and only if $\tau=\text{id}\in\mathfrak{W}_0$ and thus
\[N_{\lambda,\gamma}^{\lambda+\sigma(\gamma)}=(-1)^0m_\gamma((\lambda+\sigma(\gamma)-\lambda)=m_\gamma(\sigma(\gamma))=1.\]
\end{proof}
\par Even though the duality in $\mathcal{C}(\mathfrak{sl}_3,k)$ is clear for other reasons, its computation is straightforward from Lemma \ref{linkage}.
\begin{mycorr}\label{duality}
If $(m_1,m_2)\in\Lambda_0$, then $(m_1,m_2)^\ast=(m_2,m_1)$.
\end{mycorr}
\begin{proof}
Note that if $\mathcal{C}$ is a fusion category and $X,Y\in\mathcal{C}$ are simple, then by adjointness of duality $Y^\ast\simeq X$ if and only if
\[1=\dim_\mathbbm{k}\Hom(Y^\ast,X)=\dim_\mathbbm{k}\Hom(1,X\otimes Y).\]
Now if we denote the generating reflections $\sigma_1,\sigma_2\in\mathfrak{W}$, then \[(m_1,m_2)+(\sigma_2\sigma_1\sigma_2)(m_2,m_1)=-(m_1,m_2).\]  Thus $(m_1,m_2)+((m_1,m_2)+(\sigma_2\sigma_1\sigma_2)(m_2,m_1))=(0,0)$ and by Lemma 3.2.5, $(m_1,m_2)\otimes(m_2,m_1)$ contains $(0,0)$ with multiplicity one.
\end{proof}

\par To finish this numerical subsection we collect a formula for the multiplicative central charge of $\mathcal{C}(\mathfrak{g},k)$ \cite[Section 6.2]{DMNO} for future use.

\begin{mylemma}\label{centralcharge}
The multiplicative central charge of $\mathcal{C}:=\mathcal{C}(\mathfrak{g},k)$ is given by
\[\xi(\mathcal{C})=\exp\left(\dfrac{2\pi i}{8}\cdot\dfrac{k\dim\mathfrak{g}}{k+h^\vee}\right)\]
where $\dim\mathfrak{g}$ is the dimension of $\mathfrak{g}$ as a $\mathbb{C}$-vector space and $h^\vee$ is the dual Coxeter number of $\mathfrak{g}$
\end{mylemma}
\begin{note}
Refer to the introduction in \cite{Sawin06} for a complete list of dual Coxeter numbers.
\end{note}


\subsection{Fusion subcategories of $\mathcal{C}(\mathfrak{sl}_3,k)$}\label{subcats}

\par All fusion subcategories of $\mathcal{C}(\mathfrak{g},k)$ were classified by Sawin in Theorem 1 of \cite{Sawin06}.  For each level $k\in\mathbb{Z}_{>0}$, $\mathcal{C}(\mathfrak{sl}_3,k)$ has four fusion subcategories:
\begin{itemize}
\item[-] the trivial fusion subcategory consisting of $(0,0)$;
\item[-] the entire category $\mathcal{C}(\mathfrak{sl}_3,k)$;
\item[-] the subcategory consisting of weights $(m_1,m_2)\in\Lambda_0$ also in the root lattice.  The collection of such weights will be denoted $R_0$;
\item[-] the subcategory consisting of the weights $(0,0)$, $(k,0)$, and $(0,k)$, hereby called \emph{corner weights}.
\end{itemize}
The proof of this classification relies on two facts that will be used in the sequel.  We provide proofs here based on the original arguments found in \cite{Sawin06}, specialized to the case when $\mathfrak{g}=\mathfrak{sl}_3$ for clarity and instructive purposes.
\begin{mylemma}[Sawin]\label{dualtensor}
If a fusion subcategory $\mathcal{D}\subset\mathcal{C}(\mathfrak{sl}_3,k)$ for $k\geq2$ contains weight $\lambda$ that is not a corner weight then $\lambda\otimes\lambda^\ast$ contains $\theta$ as a direct summand.
\end{mylemma}
\begin{proof}
Since $\theta$ is self-dual by Lemma \ref{duality}, $N_{\lambda,\lambda^\ast}^\theta=N_{\lambda,\theta}^\lambda$ and by Proposition \ref{quantumracah}
\begin{equation}
N_{\lambda,\theta}^\lambda=\sum_{\tau\in\mathfrak{W}_0}(-1)^{\ell(\tau)}m_\theta(\tau(\lambda)-\lambda).\label{above1}
\end{equation}
If $\tau=\text{id}$ then the corresponding summand in (\ref{above1}) is $m_\theta(0)=2$, the rank of $\mathfrak{sl}_3$.  Now if the simple reflections $\tau_1,\tau_2,\tau_3$ are the generators of $\mathfrak{W}_0$, the reasoning leading to inequality (4) in the proof of Lemma \ref{linkage} implies if $i\neq j$
\begin{equation}
\|(\tau_i\tau_j)(\lambda)-\lambda\|>\|\tau_j(\lambda)-\lambda\|>0\label{above2}
\end{equation}
for $i,j=1,2,3$.  If $\tau_j(\lambda)-\lambda$ contributes to the sum in (\ref{above1}), then $\tau_j(\lambda)-\lambda$ must be a nonzero root.  But inequality (\ref{above2}) implies that any $\tau\in\mathfrak{W}_0$ whose reduced expression in terms of simple reflections has length greater than 1 causes $\tau(\lambda)-\lambda$ to be longer than any root, and hence does not contribute to the sum in (\ref{above1}).  Moreover, the only negative contributions to (\ref{above1}) come from simple reflections.

\par If a weight $\mu\in\Lambda_0$ is adjacent to any generating hyperplane $T_i$ for some $i=1,2,3$ ($\mu-\rho$ lies on $T_i$; see Figure \ref{fig:adjacent}), then $\|\tau_i(\mu)-\mu\|^2\leq2$ otherwise $\|\tau_i(\mu)-\mu\|^2>2$.  Thus $\tau_i(\mu)-\mu$ can contribute $-1$ to the sum in (\ref{above1}) if and only if $\mu$ is adjacent to the hyperplane $T_i$.  For $\mu\in\Lambda_0$ which are not corners, the number of adjacent generating hyperplanes adjacent to $\mu$ is at most 1, proving $N_{\lambda,\theta}^\lambda>0$.
\end{proof}
\begin{figure}[h!]
\centering
\begin{tikzpicture}[scale=2]
\fill [blue,opacity=.1] (0,0) -- (0,5*1.732/3) -- (5*1/2,2.5*1.732/3) -- cycle;

\draw[thick,->] (2,3*1.732/3) to (2.5,4.5*1.732/3) node[right] {$\|\tau_3(\mu)-\mu\|^2=2$};
\draw[thick,->] (2,3*1.732/3) to (2.75,0.75*1.732/3) node[below right] {$\|\tau_2(\mu)-\mu\|^2>2$};

\node at (0,0) {$\bullet$};
\node at (0,1.732/3) {$\bullet$};
\node at (0,2*1.732/3) {$\bullet$};
\node at (0,3*1.732/3) {$\bullet$};
\node at (0,4*1.732/3) {$\bullet$};
\node at (0,5*1.732/3) {$\bullet$};
\node at (1/2,0.5*1.732/3) {$\bullet$};
\node at (1/2,1*1.732/3+0.5*1.732/3) {$\bullet$};
\node at (1/2+0.12,1*1.732/3+0.5*1.732/3+0.195) {$\theta$};
\node at (1/2,2*1.732/3+0.5*1.732/3) {$\bullet$};
\node at (1/2,3*1.732/3+0.5*1.732/3) {$\bullet$};
\node at (1/2,4*1.732/3+0.5*1.732/3) {$\bullet$};
\node at (1,1*1.732/3) {$\bullet$};
\node at (1,2*1.732/3) {$\bullet$};
\node at (1,3*1.732/3) {$\bullet$};
\node at (1,4*1.732/3) {$\bullet$};
\node at (1.5,1*1.732/3+0.5*1.732/3) {$\bullet$};
\node at (1.5,2*1.732/3+0.5*1.732/3) {$\bullet$};
\node at (1.5,3*1.732/3+0.5*1.732/3) {$\bullet$};
\node at (2,2*1.732/3) {$\bullet$};
\node at (2,3*1.732/3) {$\bullet$};
\node at (2-0.165,3*1.732/3+0.165) {$\mu$};
\node at (2.5,2*1.732/3+0.5*1.732/3) {$\bullet$};

\foreach \i in {-1,...,2}
\foreach \j in {-1,...,4}
{
\draw[fill=black] (\i+1/2,\j*1.732/3+1.732/6) circle (0.01);
}

\foreach \i in {0,...,3}
\foreach \j in {-1,...,5}
{
\draw[fill=black] (\i,\j*1.732/3) circle (0.01);
}

\draw[->,dotted] (0,0) to (1,0) node[below right] {$\alpha_1$};
\draw[->,dotted] (0,0) to (1/2,1.732/2);
\draw[->,dotted] (0,0) to (-1/2,1.732/2) node[above left] {$\alpha_2$};

\draw[<->,dashed,thick] (-0.5,-1*1.732/3) to node[circle,draw=black,pos = 0.51,inner sep=1pt,fill=white,solid,thin] {$T_2$} (3,2.5*1.732/3);
\draw[<->,dashed,thick] (1,5*1.732/3) to node[circle,draw=black,pos = 0.15,inner sep=1pt,fill=white,solid,thin] {$T_3$} (3,3*1.732/3);

\end{tikzpicture}
    \caption{Adjacent vs. nonadjacent hyperplanes (level $k=5$)}%
    \label{fig:adjacent}%
\end{figure}
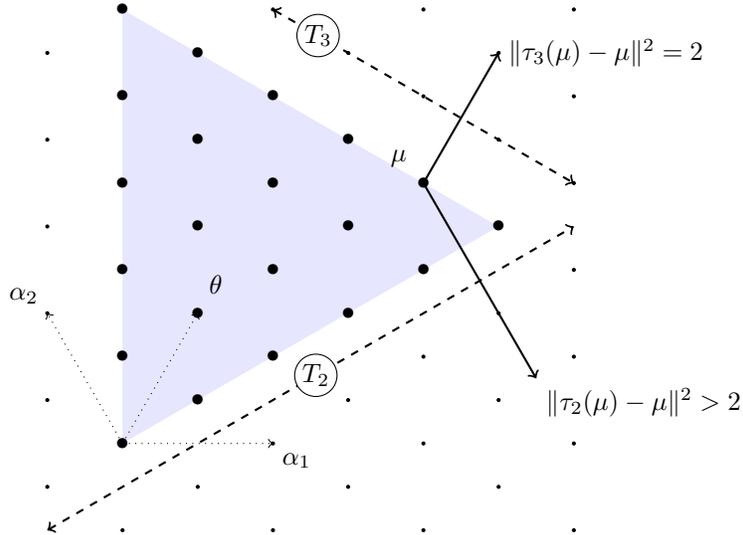
\begin{mylemma}[Sawin]\label{entirelattice}
If a fusion subcategory $\mathcal{D}\subset\mathcal{C}(\mathfrak{sl}_3,k)$ contains weight $\theta$ then $\mathcal{D}$ contains the entire root lattice in the Weyl alcove, $R_0$. 
\end{mylemma}
\begin{proof}
If $\lambda\in R_0$ then there exists a path of length $n\in\mathbb{Z}_{\geq0}$ of weights $\theta=\lambda_0,\lambda_1,\lambda_2,\ldots,\lambda_n=\lambda$ in $\Lambda_0$ such that $\lambda_{i+1}-\lambda_i$ is a root for $0\leq i\leq n-1$.  We now proceed inductively on $i$ to show each $\lambda_i$ is in $\mathcal{D}$.  Assume $\lambda_i$ is in $\mathcal{D}$ for some $0\leq i\leq n-1$.  Since $\mathfrak{W}$ acts transitively on the roots there exists $\sigma_i\in\mathfrak{W}$ such that $\sigma_i(\lambda_0)=\lambda_{i+1}-\lambda_i$.  In other words $\lambda_i+\sigma_i(\lambda_0)=\lambda_{i+1}\in\Lambda_0$ and $\lambda_0\otimes \lambda_i$ contains $\lambda_{i+1}$ as a direct summand with multiplicity 1 by Lemma \ref{linkage}.
\end{proof}


\subsection{Prime decomposition when $3\nmid k$}\label{threedoesnotdividek}

\par In light of Proposition \ref{primedecomp} the categories $\mathcal{C}(\mathfrak{sl}_3,k)$ can be decomposed into a product of prime factors which we will use in the sequel when $3\nmid k$.   
\begin{mytheorem}\label{2dk}The following are decompositions of $\mathcal{C}(\mathfrak{sl}_3,k)$ into prime factors when $3\nmid k$:
\begin{enumerate}
\item[\textnormal{(a)}]$\mathcal{C}(\mathfrak{sl}_3,1)\simeq\mathcal{C}(\mathbb{Z}/3\mathbb{Z},q_\omega)$,
\item[\textnormal{(b)}]$\mathcal{C}(\mathfrak{sl}_3,2)\simeq\mathcal{C}(\mathfrak{sl}_3,2)'_\text{pt}\boxtimes\mathcal{C}(\mathfrak{sl}_3,2)_\text{pt}\simeq(\mathcal{C}(\mathfrak{sl}_2,3)'_\text{pt})^\text{rev}\boxtimes\mathcal{C}(\mathbb{Z}/3\mathbb{Z},q_{\omega^2})$,
\end{enumerate}
and for all $m\in\mathbb{Z}_{>0}$
\begin{enumerate}
\item[\textnormal{(c)}]$\mathcal{C}(\mathfrak{sl}_3,3m+1)\simeq\mathcal{C}(\mathfrak{sl}_3,3m+1)'_\text{pt}\boxtimes\mathcal{C}(\mathfrak{sl}_3,1)$,\text{ and }
\item[\textnormal{(d)}]$\mathcal{C}(\mathfrak{sl}_3,3m+2)\simeq\mathcal{C}(\mathfrak{sl}_3,3m+2)'_\text{pt}\boxtimes\mathcal{C}(\mathfrak{sl}_3,2)_\text{pt}$.
\end{enumerate}
\end{mytheorem}
\begin{note}
Refer to Example \ref{metric} for the definitions of $q_\omega$ and $q_{\omega^2}$.
\end{note}
\begin{proof}
\par We begin by computing the twists (Section \ref{modular}) of the corner weights.  By Lemma \ref{twist},
\begin{align*}
\theta_{0,k}=\theta_{k,0}=&\exp\left(\frac{0^2+3(0)+(0)(k)+3k+k^2}{3(k+3)}\cdot2\pi i\right) \\
&=\exp\left(2k\pi i/3\right).
\end{align*}
Thus if $k\equiv1\pmod{3}$ $\theta_{0,k}=\theta_{k,0}=\omega$ and if $k\equiv2\pmod{3}$ then $\theta_{0,k}=\theta_{k,0}=\omega^2$.
\par The category $\mathcal{C}(\mathfrak{sl}_3,1)$ is pointed with 3 simple objects, and so it is determined by its twists found above.  This identifies $\mathcal{C}(\mathfrak{sl}_3,1)\simeq\mathcal{C}(\mathbb{Z}/3\mathbb{Z},q_\omega)$ which is simple, proving (a).

\par For level $k=2$ we first apply M\"uger's decomposition (Section \ref{subcategories}) and notice that $\mathcal{C}(\mathfrak{sl}_3,2)_\text{pt}\simeq\mathcal{C}(\mathbb{Z}/3\mathbb{Z},q_{\omega^2})$ based on the twist computations above. Its centralizer has 2 simple objects and is not pointed.  Thus $\mathcal{C}(\mathfrak{sl}_3,2)'_\text{pt}$ is either equivalent to $\mathcal{C}(\mathfrak{sl}_2,3)'_\text{pt}$ or $(\mathcal{C}(\mathfrak{sl}_2,3)'_\text{pt})^\text{rev}$ \cite[Section 6.4]{DMNO}\cite{ostrik}.  Using the formula found in Section 6.4 (1) of \cite{DMNO} we see
\begin{equation*}\xi\left(\mathcal{C}(\mathfrak{sl}_2,3)'_\text{pt}\right)=\exp\left(\dfrac{2\pi i}{8}\left(\dfrac{3\cdot3}{3+2}+(-1)^{(3+1)/2}\right)\right)=\exp(7\pi i/10),\end{equation*}
and thus by Lemma \ref{charge} (c),
\begin{equation*}\xi\left((\mathcal{C}(\mathfrak{sl}_2,3)'_\text{pt})^\text{rev}\right)=\exp(13\pi i/10).\end{equation*}
Using the Lemma \ref{charge} (b) we have
\begin{align*}
\xi\left(\mathcal{C}(\mathfrak{sl}_3,2)'_\text{pt}\right)&=\dfrac{\xi\left(\mathcal{C}(\mathfrak{sl}_3,2)\right)}{\xi\left(\mathcal{C}(\mathfrak{sl}_3,2)_\text{pt}\right)} \\
&=\dfrac{\exp(4\pi i/5)}{\exp(3\pi i/2)} \\
&=\exp(13\pi i/10),
\end{align*}
proving (b) since both of these categories are known to be simple.
\par The decompositions in parts (c) and (d) follow directly from M\"uger's decomposition along with parts (a) and (b), and we are left with proving simplicity.  For any $k\in\mathbb{Z}_{>0}$ the fusion subcategory of corner weights ($(0,0)$, $(k,0)$, and $(0,k)$) is pointed.  Proposition \ref{quantumracah} gives
\begin{align}
(0,k)\otimes(m_1,m_2)&=(m_2,k-m_1-m_2)\,\,\text{ and }\nonumber \\ 
(k,0)\otimes(m_1,m_2)&=(k-m_1-m_2,m_1).\label{fusionrules}
\end{align}
Thus using the well-known formula for computing the entries of the $S$-matrix \cite[Proposition 8.13.8]{tcat}
\begin{align*}
s_{(0,k),(m_1,m_2)}&=\dfrac{\theta_{m_2,3-m_1-m_2}}{\theta_{m_1,m_2}}=\exp\left(\dfrac{1}{3}(k-2m_1-m_2)\cdot2\pi i\right)\text{ and} \\
s_{(k,0),(m_1,m_2)}&=\dfrac{\theta_{3-m_1-m_2,m_1}}{\theta_{m_1,m_2}}=\exp\left(\dfrac{1}{3}(k-m_1-2m_2)\cdot2\pi i\right).
\end{align*}
This implies $s_{(0,k),(m_1,m_2)}=s_{(k,0),(m_1,m_2)}=1$ if and only if $m_1\equiv m_2\pmod{3}$, that is to say $(m_1,m_2)\in R_0$.  Having $s_{X,Y}=1$ is one of several equivalent conditions for $X$ and $Y$ to centralize one another \cite[Proposition 8.20.5]{tcat}.  Moreover if $3\nmid k$ then the corners $(0,k)$ and $(k,0)$ are not in the root lattice so by Sawin's classification of fusion subcategories (Section \ref{subcats}), the centralizers of the pointed subcategories are simple and thus prime.

\end{proof}

\par We now take a moment to compute the central charge of $\mathcal{C}(\mathfrak{sl}_3,k)'_\text{pt}$ for future use when $k=3m+1$ or $k=3m+2$ with $m\in\mathbb{Z}_{\geq0}$ based on (c), (d) of Proposition \ref{2dk} and the multiplicativity of $\xi$.

\begin{mycorr}\label{nonintegercharge}For $m\in\mathbb{Z}_{\geq0}$
\begin{enumerate}
\item[\textnormal{(a)}] when $k=3m+1$ ($m\neq0$)
\[\xi\left(\mathcal{C}(\mathfrak{sl}_3,k)'_\text{pt}\right)=\exp\left(\dfrac{9m}{6m+8}\pi i\right),\]
\item[\textnormal{(b)}] and when $k=3m+2$
\[\xi\left(\mathcal{C}(\mathfrak{sl}_3,k)'_\text{pt}\right)=\exp\left(\dfrac{3m-7}{6m+10}\pi i\right).\]
\end{enumerate}
\end{mycorr}


\subsection{Simplicity of $\mathcal{C}(\mathfrak{sl}_3,k)_A^0$ when $3\mid k$}\label{threedividesk}

\par When $k=3m$ for some $m\in\mathbb{Z}_{>0}$, the object $A=(0,0)\oplus(3m,0)\oplus(0,3m)$ has the structure of a connected \'etale algebra and so we can consider the nondegenerate braided fusion category consisting of dyslectic $A$-modules $\mathcal{C}_A^0:=\mathcal{C}(\mathfrak{sl}_3,3m)_A^0$ (Section \ref{etale}).  The act of tensoring with $(3m,0)$ or $(0,3m)$ combinatorally results in a rotation of $\Lambda_0$ by 120 degrees counter-clockwise or clockwise, respectively, as illustrated in Figure \ref{fig:example5}.

\par There are two types of simple objects in $\mathcal{C}_A^0$:
\begin{itemize}
\item[-]\emph{Free} objects (Section 2.3) are of the form $F(\lambda)=\lambda\otimes A$ for $\lambda\in R_0$ not equal to $(m,m)$.  These objects are the sum of the objects in orbits of size three under the 120 degree rotations described above.
\item[-]Three \emph{stationary} objects are isomorphic to $(m,m)$ as objects of $\mathcal{C}(\mathfrak{sl}_3,3m)$, but non-isomorphic as $A$-modules.  If $\rho_1:(m,m)\otimes A\longrightarrow(m,m)$ is the action on one of these $A$-modules then the others have actions given by $\rho_\omega=\omega\rho_1$ and $\rho_{\omega^2}=\omega^2\rho_1$ where $\omega=\exp(2\pi i/3)$.
\end{itemize}

\par Denote any of these three stationary objects as $X\in\mathcal{C}^0_A$ or collectively as $X_1,X_2,X_3\in\mathcal{C}^0_A$.  At no point in what follows will it become important to distinguish their $A$-module structures and in fact doing so can lead to ambiguity in computations as illustrated in the $\mathfrak{sl}_2$ case described in Section 7 of \cite{KiO}.

\begin{example}\label{kequalsthree}
 When $k=3$ the only free object is the identity $F(0,0)$ and there are three stationary objects $X_1,X_2,X_3$ corresponding to the central weight $(1,1)$.  This category is pointed by Theorem 1.18 of \cite{KiO} which states that for $i=1,2,3$
\begin{equation}
\dim(X_i)=\dfrac{\dim(1,1)}{\dim(A)}=\dfrac{\sin^2\left(\dfrac{\pi}{3}\right)\sin\left(\dfrac{2\pi}{3}\right)}{3\sin\left(\dfrac{\pi}{3}\right)\sin^2\left(\dfrac{\pi}{6}\right)}=1.
\end{equation}
The simple objects of $\mathcal{C}_A^0$ form an abelian group of order four, which is either cyclic or the Klein-4 group.  But since the act of tensoring by $(0,3)$ and $(3,0)$ cyclically permutes $X_1,X_2,X_3$ we must have $\mathcal{C}_A^0\simeq\mathcal{C}(\mathbb{Z}/2\oplus\mathbb{Z}/2\mathbb{Z},q)$ with quadratic form $q:\mathbb{Z}/2\mathbb{Z}\oplus\mathbb{Z}/2\mathbb{Z}\longrightarrow\mathbb{C}^\times$ which is 1 on the unit object and $-1$ on the stationary objects.  This category is evidently not simple as $\mathbb{Z}/2\mathbb{Z}\oplus\mathbb{Z}/2\mathbb{Z}$ has many subgroups.
\end{example}

\begin{example}\label{ex:kequalssix}
We will examine the case of $k=6$ as it is also of great interest.  There are three stationary objects $X_1$, $X_2$, and $X_3$, and three free objects $Y_1=F(0,0)$, $Y_2=F(1,1)$, and $Y_3=F(3,3)$.  The tensor structure of the free module functor gives the fusion rules between the free objects:
\begin{align*}
Y_2\otimes_A Y_2&=Y_1\oplus2Y_2\oplus2Y_3\oplus X_1\oplus X_2\oplus X_3,& \\
Y_2\otimes_A Y_3&=2Y_2\oplus Y_3\oplus X_1\oplus X_2\oplus X_3,& \text{ and}\\
Y_3\otimes_A Y_3&=Y_1\oplus Y_2\oplus Y_3\oplus X_1\oplus X_2\oplus X_3.& \\
\end{align*}
For instance
\begin{align*}
Y_2\otimes_A Y_2&=F(1,1)\otimes_A F(1,1) \\
			&=F((1,1)\otimes(1,1)) \\
			&=F((0,0)\oplus(0,3)\oplus(3,0)\oplus(1,1)\oplus(1,1)\oplus(2,2)) \\
&=Y_1\oplus Y_3\oplus Y_3\oplus Y_2\oplus Y_2\oplus F(2,2) \\
&=Y_1\oplus2Y_2\oplus2Y_3\oplus X_1\oplus X_2\oplus X_3.
\end{align*}
Now to compute the remaining fusion rules notice that each $X_i$ is self dual since the act of tensoring with corner weights $(6,0)$ and $(0,6)$ cyclically permute the $X_i$.  Since it is evident at least one $X_i$ must be self dual, the orbit of this self dual object under the aforementioned cyclic permutation must also be self dual.

\par By comparing dimensions we must have
\begin{align}
Y_2\otimes_A X_i&=Y_2\oplus Y_3\oplus X_j\oplus X_k, \label{y2x}\\
Y_3\otimes_A X_i&=Y_2\oplus Y_3\oplus X_\ell,\label{y3x} 
\end{align}
and
\begin{equation}\label{fusionX}
X_r\otimes_A X_s=\left\{\begin{array}{lcr}
Y_1\oplus Y_3\oplus X_t & \text{ if } & r=s \\
Y_2\oplus X_u & \text{ if } & r\neq s
\end{array}\right.
\end{equation}
for some $j,k,\ell,t,u=1,2,3$.  We will now determine the unknown summands in (\ref{y2x}), (\ref{y3x}), and (\ref{fusionX}).  For instance since all objects are self dual, if $i\neq j$ by (\ref{fusionX})
\begin{equation*}
1=\dim_\mathbb{C}\Hom(Y_2,X_i\otimes_AX_j)=\dim_\mathbb{C}\Hom(X_i,Y_2\otimes_A X_j).
\end{equation*}
Hence $i,j,k$ are all distinct in (\ref{y2x}).  Similarly
\begin{equation*}
0=\dim_\mathbb{C}\Hom(Y_3,X_i\otimes_AX_j)=\dim_\mathbb{C}\Hom(X_i,Y_3\otimes_A X_j),
\end{equation*}
which implies $i=\ell$ in (\ref{y3x}) above.  For any $i,j=1,2,3$ denote the unknown summand in $X_i\otimes X_j$ by $X_{i,j}$.  We will show that if $X_i\neq X_{i,i}$ then the following equality cannot hold:
\begin{equation}3=\dim_\mathbb{C}\Hom(X_i\otimes_AX_i,X_i\otimes_AX_i)=\dim_\mathbb{C}\Hom(Y_1,X_i^{\otimes_A4}).\label{cont}\end{equation}
To see the contradiction note that $X_{i,i}^{\otimes_A3}=2Y_2\oplus Y_3\oplus2X_i\oplus X_{i,ii}$ where $X_{i,ii}$ is the unknown summand in the product $X_i\otimes X_{i,i}$.  Since our assumption implies $X_i\neq X_{i,i}^\ast$ then $X_i\neq X_{i,ii}$.  But this would imply $\dim_\mathbb{C}\Hom(X_i\otimes_AX_i,X_i\otimes_AX_i)=2$ by the above computation of $X_i^{\otimes_A3}$, contradicting (\ref{cont}).
\par Now to determine the remaining fusion rule in (\ref{fusionX}), computing $Y_2\otimes(X_i\otimes X_i)$ and $(Y_2\otimes X_i)\otimes X_i$ using (\ref{y2x}), (\ref{y3x}), and the first part of (\ref{fusionX}) shows that $X_{i,i}=X_i$, $X_{i,j}$, and $X_{j,k}$ are distinct.  By symmetry of this computation $X_{j,i}$, $X_{j,j}=X_j$, and $X_{j,k}$ as well as $X_{k,i}$, $X_{k,j}$, and $X_{k,k}=X_k$ are distinct triples as well.  This immeditely proves that $X_i\otimes X_j=Y_2\oplus X_k$ where $i,j,k$ are distinct and the fusion rules are computed completely.

\begin{note} $\mathcal{C}(\mathfrak{sl}_3,6)_A^0$ is simple.\end{note}

\par The $S$-matrix is now computed using Proposition 8.13.8 of \cite{tcat} in conjunction with the above fusion rules and the numerical data found in Section \ref{sec:moddata} to yield
\begin{equation*}
S=\left[\begin{array}{cccccc}
1 & \zeta+1 & \zeta & \epsilon & \epsilon & \epsilon \\
\zeta+1 & \zeta & -1 & -\epsilon & -\epsilon & -\epsilon \\
\zeta & -1 & -(\zeta+1) & \epsilon & \epsilon & \epsilon \\
\epsilon & -\epsilon & \epsilon & 2\epsilon & -\epsilon &-\epsilon  \\
\epsilon & -\epsilon & \epsilon &-\epsilon  & 2\epsilon &-\epsilon  \\
\epsilon & -\epsilon & \epsilon &-\epsilon  &-\epsilon  & 2\epsilon
\end{array}\right],
\end{equation*}
where $\zeta$ is the positive root of $x^3-3x^2-6x-1$ and $\epsilon$ is the greatest positive root of $x^3-3x^2+1$.  The $T$-matrix for $\mathcal{C}_A^0$ contains the same twists as the corresponding objects in $\mathcal{C}(\mathfrak{sl}_3,6)$:
\begin{equation*}
T=\left[\begin{array}{cccccc}
1 & 0 & 0 & 0 & 0 & 0 \\
0 & \omega & 0 & 0 & 0 & 0 \\
0 & 0 & \omega^2 & 0 & 0 & 0 \\
0 & 0 & 0 & \eta & 0 &0  \\
0 & 0 & 0 &0  & \eta &0  \\
0 & 0 & 0 &0  &0  & \eta
\end{array}\right],
\end{equation*}
where $\omega=\exp(2\pi i/3)$ and $\eta=\exp\left(2\pi i/9\right)$.

\begin{note}
This $S$-matrix was computed independently by Daniel Creamer applying algebro-geometric methods to the admissability criterion found for example in \cite{rowell} under the assumption that this category was self dual.  The computational software \emph{Kac} can be used to approximate the above $S$-matrix.
\end{note}
\begin{figure}[h!]
\centering
\begin{tikzpicture}[scale=1.5]
\fill [blue,opacity=.05] (0,0) -- (0,6*1.732/3) -- (3,3*1.732/3) -- cycle;

\draw[dotted,thick] (0,0) -- (0,6*1.732/3) -- (3,3*1.732/3) -- cycle;
\draw[dotted,thick] (1/2,1.732/3+0.5*1.732/3) -- (2,3*1.732/3) -- (1/2,4*1.732/3+0.5*1.732/3) -- cycle;
\draw[dotted, very thick] (0,3*1.732/3) -- (1.5,4*1.732/3+0.5*1.732/3) -- (1.5,1*1.732/3+0.5*1.732/3) -- cycle;

\node (c) at (0,0) {};
\node (b) at (0,6*1.732/3) {};
\node at (1,3*1.732/3) {\Large$\triangleright$};
\node (a) at (3,3*1.732/3) {};

\draw[thin] (3,5.3*1.732/3) rectangle (6,6.7*1.732/3);
\node at (4.5,5.666*1.732/3) {$\triangleright$  Stationary Objects};
\draw[dotted] (3.5,6.333*1.732/3) to (4,6.333*1.732/3);
\node at (4.93,6.333*1.732/3) {Free Objects};
\draw[bend right,->,out=-50,in=235] (a) edge node[fill=white,pos = 0.5] {$\otimes(6,0)$}  (b);
\draw[bend right,->,out=-50,in=235] (b) edge node[fill=white,pos = 0.5] {$\otimes(6,0)$}  (c);
\draw[bend right,->,out=-50,in=235] (c) edge node[fill=white,pos = 0.5] {$\otimes(6,0)$} (a);
\end{tikzpicture}
    \caption{$\mathcal{C}_A^0$ at level $k=6$, and the action of a corner weight}%
    \label{fig:example5}%
\end{figure}
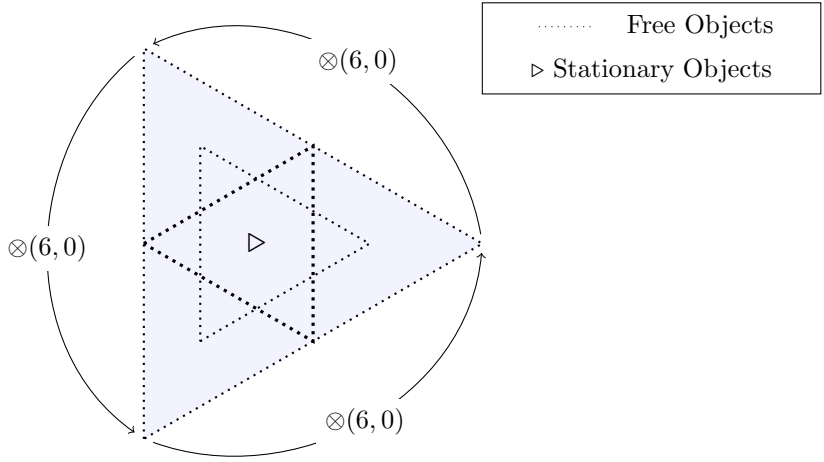
\end{example}

\begin{theo}\label{maintheorem}
The categories $\mathcal{C}_A^0:=\mathcal{C}(\mathfrak{sl}_3,3m)_A^0$ are simple for $m\geq2$.
\end{theo}

\begin{proof}
\par Assume that $\mathcal{D}\subset\mathcal{C}_A^0$ is a fusion subcategory containing a non-trivial simple  free object $F(\lambda)$ for some $\lambda\in R_0$.  As noted in Section \ref{subcategories} the fusion subcategory $\mathcal{D}$ must also contain $F(\lambda)^\ast$.  Lemma \ref{dualtensor} implies $\lambda\otimes \lambda^\ast$ contains $\theta$ as a summand.  So by the tensor structure of $F$ (Section \ref{etale}),
\[F(\lambda)\otimes_A F(\lambda)^\ast=F(\lambda)\otimes_A F(\lambda^\ast)= F(\lambda\otimes\lambda^\ast)\] which implies $F(\lambda)\otimes_A F(\lambda)^\ast$ contains $F(\theta)$ as a summand.  Finally Lemma \ref{entirelattice} implies that there exists an $n\in\mathbb{Z}_{>0}$ such that $\theta^n$ contains $\mu$ as a direct summand for any  $\mu\in R_0$.  Hence by the above argument using the tensor structure of $F$, $F(\theta)^n$ will contain $F(\mu)$ as a direct summand.  In this case we have proven $\mathcal{D}=\mathcal{C}_A^0$ since all simple objects are direct summands of free objects (Section \ref{etale}).

\par The only case that remains is if the fusion subcategory $\mathcal{D}$ only contains stationary object(s) $X\in\mathcal{C}_A^0$ corresponding to the central weight $(m,m)$ which we will denote as $\nu$ for brevity.

\begin{mylemma}\label{mplusone}
In $\mathcal{C}(\mathfrak{sl}_3,3m)$ with $m\in\mathbb{Z}_{>0}$, we have $N_{\nu,\nu}^\nu=m+1$.
\end{mylemma}

\begin{proof}
\par Proposition \ref{quantumracah} gives
\begin{equation}N_{\nu,\nu}^\nu=\sum_{\tau\in\mathfrak{W}_0}(-1)^\tau m_\nu(\tau(\nu)-\nu)\label{qabove}.\end{equation}
For the simple reflections $\tau_1,\tau_2,\tau_3\in\mathfrak{W}_0$, $\|\tau_i(\nu)-\nu\|>\|\nu\|$ and by the reasoning leading to inequality (\ref{dagger2}) in the proof of Lemma \ref{linkage} only simple reflections can contribute to the sum in (\ref{qabove}).  Hence the only nonzero term in (\ref{qabove}) comes from the identity in $\mathfrak{W}_0$ and thus $N_{\nu,\nu}^\nu=m_\nu(0)$.  If $p(\mu)$ is the number of ways of writing a weight $\mu$ as a sum of positive roots, by Kostant's multiplicity formula \cite{Kostant58}
\begin{align*}
m_\nu(0)&=\sum_{\sigma\in\mathfrak{W}}(-1)^{\ell(\sigma)}p(\sigma((m+1)\alpha_1+(m+1)\alpha_2)-\alpha_1-\alpha_2) \\
		&=p(m\alpha_1+m\alpha_2)
\end{align*}
because the argument of $p$ is not dominant for any nontrivial elements of the Weyl group.  Now it suffices to note that since there are three positive roots, $\alpha_1,\alpha_2,\alpha_1+\alpha_2$, then $p(m\alpha_1+m\alpha_2)=m+1$ because to count the number of ways to write $m\alpha_1+m\alpha_2$ as a sum of positive roots is the same as choosing how many copies of $\alpha_1+\alpha_2$ to use (the number of $\alpha_1$'s and $\alpha_2$'s are then determined).
\end{proof}

\par To finish the proof of Theorem \ref{maintheorem}, we wish to show that some nontrivial simple free object appears as a summand of $X\otimes_AX$ ($\1$ is a summand of $X\otimes_AX^\ast$ as $\dim_\mathbb{C}\Hom(\1,X\otimes_AX^\ast)=\dim_\mathbb{C}\Hom(X,X)=1$). We have already shown above that this would imply the nontrivial simple free summand generates the entire category $\mathcal{C}_A^0$.

\par If $X\otimes_AX^\ast$ does \emph{not} contain a simple non-trivial summand different from $X$ then we must have
\[X\otimes_AX^\ast=\1\oplus nX\]
where $n\in\mathbb{Z}_{\geq0}$ and $n\leq m+1$ by Lemma \ref{mplusone}. Using the additivity and multiplicativity of dimension the above implies
\begin{equation*}\dim(X)^2-n\dim(X)-1=0,\end{equation*}
hence
\begin{equation}\dim(X)=\frac{n+\sqrt{n^2+4}}{2}\leq m+3.\label{inequal}\end{equation}
By Proposition \ref{prop:dim} and Theorem 1.18 of \cite{KiO},
\[\dim(X)=\dfrac{\sin^2\left(\dfrac{(m+1)\pi}{3(m+1)}\right)\sin\left(\dfrac{(2(m+1)\pi}{3(m+1)}\right)}{3\sin\left(\dfrac{2\pi}{3(m+1)}\right)\sin^2\left(\dfrac{\pi}{3(m+1)}\right)}=\dfrac{\sqrt{3}}{8\sin\left(\dfrac{2\pi}{3(m+1)}\right)\sin^2\left(\dfrac{\pi}{3(m+1)}\right)}.\]
But for $m\geq2$ the arguments of the above sines are are positive hence $\sin(x)<x$ and we have
\begin{equation*}\dim(X)>\dfrac{27\sqrt{3}(m+1)^3}{16\pi^3}\end{equation*}
which is strictly greater than $m+3$ for $m\geq3$, contradicting the inequality in (\ref{inequal}).  The case when $m=2$ was described explicitly in Example \ref{ex:kequalssix}.
\end{proof}


\section{Witt group relations}

\subsection{Modular Invariants and conformal embeddings}\label{sec:modcon}

\par Given a connected \'etale algebra $A$ in a modular tensor category $\mathcal{C}$ one can construct $Z_A\in\text{Mat}_n(\mathbb{Z}_{\geq0})$, the (symmetric) modular invariant associated with $A\in\mathcal{C}$ where $n=|\mathcal{O}(\mathcal{C})|$.  The matrix $Z_A$ commutes with the modular group action associated with the modular tensor category $\mathcal{C}$.  Refer to \cite{ostrik2003} for a detailed explanation of the passage from algebras to modular invariants.

\par As the rank of $\mathfrak{g}$ and the level $k$ increase without bound across all $\mathcal{C}(\mathfrak{g},k)$, computing and classifying these modular invariants from a numerical standpoint becomes an arduous task.  A complete and rigorous classification only exists for all levels $k$ in the $\mathfrak{g}=\mathfrak{sl}_2$ \cite{CIZ} and $\mathfrak{g}=\mathfrak{sl}_3$ \cite{gannon} cases, the latter being of particular interest to our second main result.

\begin{theorem}[Gannon]\label{gannon}
There are three familes of symmetric modular invariants for $\mathfrak{sl}_3$:
\begin{enumerate}
\item[] (Type A) For all levels $k$, the modular invariant associated with the trivial connected \'etale algebra $(0,0)$;
\item[] (Type D) For levels $k$ such that $3\mid k$, the modular invariant associated with the connected \'etale algebra $(0,0)\oplus(k,0)\oplus(0,k)$;
\item[] (Type E) At levels $k=5,9,21$, the modular invariants associated with exceptional \'etale algebras.
\end{enumerate}
\end{theorem}

\par Translating this into our discussion of Witt class representatives and the decompositions/reductions found in Sections \ref{threedoesnotdividek} and \ref{threedividesk}, we have the following corollary.
\begin{mycorr}\label{anisotropic}
The categories $\mathcal{C}(\mathfrak{sl}_3,3m+1)'_\text{pt}$ and $\mathcal{C}(\mathfrak{sl}_3,3m+2)'_\text{pt}$ are completely anisotropic for $m\in\mathbb{Z}_{>0}$ and $3m+2\neq5$, while $\mathcal{C}_A(\mathfrak{sl}_3,3m)_A^0$ is completely anisotropic for $m\in\mathbb{Z}_{\geq2}$ and $3m\neq9,21$.
\end{mycorr}

\begin{note}This corollary depends on the classification of connected \'etale algebras in tensor products of nondegenerate braided fusion categories found in \cite{DNO}.\end{note}

\par In the cases $k=5,9,21$ we will use alternative methods to identify completely anisotropic representatives of the Witt classes of $\mathcal{C}(\mathfrak{sl}_3,k)$ for $k=5,9,21$.  In particular the theory of \emph{conformal embeddings} can be used to construct relations among the classes $[\mathcal{C}(\mathfrak{g},k)]$ for any finite dimensional simple Lie algebra $\mathfrak{g}$ \cite[Section 6.2]{DMNO}.  A complete classification of such conformal embeddings is given in \cite{bb} and \cite{sw}.

\par Each conformal embedding $\mathfrak{g}\subset\mathfrak{g}'$ gives rise to equivalences of the form $[\mathcal{C}(\mathfrak{g},k)]=[\mathcal{C}(\mathfrak{g}',k')]$ for some levels $k,k'\in\mathbb{Z}_{>0}$.  Three conformal embeddings are of interest for the classification of $\mathfrak{sl}_3$ relations: $A_{2,9}\subseteq E_{6,1}$, $A_{2,21}\subseteq E_{7,1}$, and $A_{2,5}\subseteq A_{5,1}$.

\par The category $\mathcal{C}(E_6,1)$ is pointed with three simple objects.  Using Proposition \ref{centralcharge} we find
\[\xi\left(\mathcal{C}(E_6,1)\right)=\exp\left(\dfrac{2\pi i}{8}\cdot\dfrac{1\cdot78}{1+12}\right)=-i.\]
Since pointed categories $\mathcal{C}(\mathbb{Z}/3\mathbb{Z},q)$ are determined by their central charge, we conclude
\[\mathcal{C}(E_6,1)\simeq\mathcal{C}(\mathbb{Z}/3\mathbb{Z},q_{\omega^2})\simeq\mathcal{C}(\mathfrak{sl}_3,2)_\text{pt},\]
which is simple and completely anisotropic.  Moreover
\[[\mathcal{C}(\mathfrak{sl}_{3},9)]=[\mathcal{C}(\mathfrak{sl}_3,2)_\text{pt}].\]

\par Similarly the category $\mathcal{C}(E_7,1)$ is pointed with two simple objects.  Using Proposition \ref{centralcharge} we find
\[\xi\left(\mathcal{C}(E_7,1)\right)=\exp\left(\dfrac{2\pi i}{8}\cdot\dfrac{1\cdot133}{1+18}\right)=\dfrac{1-i}{\sqrt{2}}.\]
Since pointed categories $\mathcal{C}(\mathbb{Z}/2\mathbb{Z},q)$ are determined by their central charge, we conclude
\[\mathcal{C}(E_7,1)\simeq\mathcal{C}(\mathbb{Z}/2\mathbb{Z},q_-),\]
where $q_-(1)=-i$, which is simple and completely anisotropic.  From Proposition \ref{centralcharge} we have
\[\xi\left(\mathcal{C}(\mathfrak{sl}_2,1)\right)=\exp\left(\dfrac{2\pi i}{8}\cdot\dfrac{1\cdot3}{1+2}\right)=\dfrac{1+i}{\sqrt{2}}.\]
From Lemma \ref{charge} (c) this implies
\[\xi\left(\mathcal{C}(\mathfrak{sl}_2,1)^\text{rev}\right)=\dfrac{1-i}{\sqrt{2}}\]
and moreover we conclude
\begin{equation}[\mathcal{C}(\mathfrak{sl}_{3},21)]=[(\mathcal{C}(\mathfrak{sl}_2,1)^\text{rev}].\label{level21}\end{equation}
Lastly $\mathcal{C}(\mathfrak{sl}_5,1)$ is pointed with 5 simple objects.  As noted in Example 6.2 of \cite{DMNO}, $\mathcal{C}(\mathfrak{sl}_n,1)\simeq\mathcal{C}(\mathbb{Z}/n\mathbb{Z},q)$ where $q(\ell)=\exp\left(\pi i\ell^2(n-1)/n)\right)$ and hence
\[[\mathcal{C}(\mathfrak{sl}_3,5)]=[\mathcal{C}(\mathfrak{sl}_5,1)]=[\mathcal{C}(\mathbb{Z}/5\mathbb{Z},q)].\]


\subsection{A Classification of Relations}\label{sec:class}

\begin{theo}\label{second}
The only relations in the Witt group of nondegenerate braided fusion categories $\mathcal{W}$ coming from the subgroup generated by $[\mathcal{C}(\mathfrak{sl}_3,k)]$ are the following:
\begin{align*}
\textnormal{(a)}\hspace{5 mm}&[\mathcal{C}(\mathfrak{sl}_3,1)]^4=[\Vep], \\
\textnormal{(b)}\hspace{5 mm}&[\mathcal{C}(\mathfrak{sl}_3,3)]^2=[\Vep], \\
\textnormal{(c)}\hspace{5 mm}&[\mathcal{C}(\mathfrak{sl}_3,5)]^2=[\Vep], \\
\textnormal{(d)}\hspace{5 mm}&[\mathcal{C}(\mathfrak{sl}_3,1)]^3=[\mathcal{C}(\mathfrak{sl}_3,9)],\text{ and}\\
\textnormal{(e)}\hspace{5 mm}&[\mathcal{C}(\mathfrak{sl}_3,21)]^8=[\Vep].
\end{align*}
\end{theo}

\begin{proof}
\par Our approach to this proof will be to show that the above relations hold, and then prove that these are the only relations which can exist by identifying the unique representatives of each class $[\mathcal{C}(\mathfrak{sl}_3,k)]$ as described in Section \ref{witt}.
\par Since $\mathcal{C}(\mathfrak{sl}_3,k)$ in equations (a)--(e) above are all Witt equivalent to a pointed modular tensor category by the computations in Proposition \ref{2dk} (a), Example \ref{kequalsthree}, and Section \ref{sec:modcon}, the relations follow from the exposition in Appendix A.7 of \cite{DGNO} which explicitly describes the pointed subgroup $\mathcal{W}_\text{pt}\subset\mathcal{W}$.  The remaining question is whether these relations are exhaustive.

\par By Proposition \ref{2dk}, Theorem \ref{maintheorem}, and Corollary \ref{anisotropic} for $m\in\mathbb{Z}_{\geq0}$ we have collected simple, completely anisotropic, nondegenerate braided fusion categories
\begin{align}
&\mathcal{C}(\mathfrak{sl}_3,3m+1)'_\text{pt}, &&\text{ for }m\neq0 \label{1}\\
&\mathcal{C}(\mathfrak{sl}_3,3m+2)'_\text{pt}, &&\text{ for }m\neq1 \label{2}\\
&\mathcal{C}(\mathfrak{sl}_3,3m)^0_A, &&\text{ for }m\neq0,1,3,7\label{3}
&\end{align}
We claim the categories in the above families are not equivalent and will prove this by noting their central charges are distinct using Lemma \ref{centralcharge} and Lemma \ref{nonintegercharge}.  For $j=\ref{1},\ref{2},\ref{3}$ and admissible $m\in\mathbb{Z}_{\geq0}$ as in lists $(\ref{1})$ through $(\ref{3})$ above let $\lambda_j(m)$ be such that $\xi(\mathcal{C}_{(j)}(m))=\exp(\lambda_{(j)}(m)\pi i)$ where $\mathcal{C}_{(j)}(m)$ is a category in list $(j)$ for the given $m\in\mathbb{Z}_{\geq0}$.  For any admissible $r,s,t\in\mathbb{Z}_{>2}$,
\begin{equation}2>\lambda_{(\ref{3})}(r)\geq 3/2>\lambda_{(\ref{1})}(s)>1>1/2>\lambda_{(\ref{2})}(t)>0.\label{tabel}\end{equation}
For $m\leq2$ we have the following table:
\begin{equation}\label{table}
\begin{array}{|c|c|c|c|}
\hline & \lambda_{(\ref{1})}(m) & \lambda_{(\ref{2})}(m) & \lambda_{(\ref{3})}(m) \\\hline
m=0 & \text{n/a} & -7/10&\text{n/a} \\\hline
m=1 & 9/14&\text{n/a} & \text{n/a}\\\hline
m=2 & 9/10& -1/22& 4/3\\\hline
\end{array}
\end{equation}

\par Recall the \emph{Witt group of slightly degenerate braided fusion categories}, $s\mathcal{W}$, introduced in \cite{DNO}.  Since slightly degenerate braided fusion categories admit a unique decomposition into $s$-simple components \cite[Definition 4.9, Theorem 4.13]{DNO}, there are no nontrivial relations in $s\mathcal{W}$ other than relations of the form $[\mathcal{C}]=[\mathcal{C}]^{-1}$ \cite[Remark 5.11]{DNO}.  The categories in (\ref{1})--(\ref{3}) are simple and unpointed, hence their image under the group homomorphism $S:\mathcal{W}\longrightarrow s\mathcal{W}$ \cite[Section 5.3]{DNO} is $s$-simple.  Their image is also completely anisotropic and slightly degenerate.  Hence any nontrivial relation in $\mathcal{W}$ between these categories would pass to a relation in $s\mathcal{W}$ under the map $S$ and this relation is nontrivial provided they are not in the kernel of $S$, which is $\mathcal{W}_\text{Ising}$ consisting of the Ising braided categories.

\par When $[\mathcal{C}]=[\mathcal{C}]^{-1}$ in $s\mathcal{W}$, $\mathcal{C}\simeq\mathcal{C}^\text{rev}$ which implies $\xi(\mathcal{C})=\pm1$.  This cannot be true since the inequalities for the central charges found above and (20) show that the central charge of the categories in (\ref{1}), (\ref{2}), and (\ref{3}) is never $\pm1$.  Thus the relations (a)--(e) are exhaustive.
\end{proof}

\par Once relations in the subgroups individually generated by $[\mathcal{C}(\mathfrak{sl}_2,k)]$ and $[\mathcal{C}(\mathfrak{sl}_3,k)]$ are classified one should then classify all relations between the two families.  To organize the search for these relations we will proceed in two stages: first we consider coincidences between the $\mathfrak{sl}_3$ relations from Theorem \ref{second} and those $\mathfrak{sl}_2$ relations found in \cite[Section 5.5]{DNO}; secondly compare the lists of simple completely anisotropic Witt class representatives for $\mathfrak{sl}_3$ found in (\ref{1}) through (\ref{3}) above and those for $\mathfrak{sl}_2$ found in \cite[Section 5.5]{DNO}.

\par For the first stage we have $[\mathcal{C}(\mathfrak{sl}_3,3)]=[\mathcal{C}(\mathfrak{sl}_3,3)_A^0]$ is in $\mathcal{W}_\text{Ising}$, a cyclic group of order 16 which is generated by $[\mathcal{C}(\mathfrak{sl}_2,2)]$.  Hence we have the relation
\begin{equation}[\mathcal{C}(\mathfrak{sl}_3,3)]=[\mathcal{C}(\mathfrak{sl}_2,2)]^8,\label{16}\end{equation}
and multiplying both sides above by $[\mathcal{C}(\mathfrak{sl}_2,2)]^{11}$,
\begin{equation}[\mathcal{C}(\mathfrak{sl}_3,3)][\mathcal{C}(\mathfrak{sl}_2,2)]^{11}=[\mathcal{C}(\mathfrak{sl}_2,6)]^2\end{equation}
using \cite[Section 5.5 (36)]{DNO}. Also Equation (\ref{level21}) from Section \ref{sec:modcon} implies
\begin{equation}[\mathcal{C}(\mathfrak{sl}_3,21)][\mathcal{C}(\mathfrak{sl}_2,1)]=[\Vep].\label{17}\end{equation}
\par The relation implied by Proposition \ref{2dk} (b) is
\begin{equation*}[\mathcal{C}(\mathfrak{sl}_3,2)][\mathcal{C}(\mathfrak{sl}_2,3)'_\text{pt}]=[\mathcal{C}(\mathfrak{sl}_3,2)_\text{pt}].\end{equation*}
But by \cite[Section 5.5 (32)]{DNO} $[\mathcal{C}(\mathfrak{sl}_2,28)]=[\mathcal{C}(\mathfrak{sl}_2,3)'_\text{pt}]=[\mathcal{C}(\mathfrak{sl}_2,3)][\mathcal{C}(\mathfrak{sl}_2,1)]^{-1}$, and $[\mathcal{C}(\mathfrak{sl}_3,2)_\text{pt}]=[\mathcal{C}(\mathfrak{sl}_3,9)]$ by Section \ref{sec:modcon}.  Hence
\begin{equation}[\mathcal{C}(\mathfrak{sl}_3,2)][\mathcal{C}(\mathfrak{sl}_2,28)]=[\mathcal{C}(\mathfrak{sl}_3,9)].\label{2and28}\end{equation}
The relation from \cite[Section 5.5 (33)]{DNO} is very similar to Theorem \ref{second} (a) and (c), and this is no coincidence since there exists a conformal embedding $A_{1,4}\subset A_{3,1}$ (see the aforementioned Appendix of \cite{DMNO}) hence
\begin{equation}[\mathcal{C}(\mathfrak{sl}_2,4)]=[\mathcal{C}(\mathfrak{sl}_3,1)],\end{equation}
and by Theorem \ref{second} (d)
\begin{equation}[\mathcal{C}(\mathfrak{sl}_2,4)]^3=[\mathcal{C}(\mathfrak{sl}_3,9)].\label{4and1}\end{equation}
\par We have exhaustively compared the relations from $\mathfrak{sl}_2$ and $\mathfrak{sl}_3$ so we can proceed to stage two where we check if categories listed in (\ref{1}) through (\ref{3}) are equivalent to any of those simple completely anisotropic representatives coming from $\mathfrak{sl}_2$ listed in \cite[Section 5.5]{DNO} (or their reverse categories).  The first deduction can be made by noting that the categories $\mathcal{C}(\mathfrak{sl}_2,2\ell+1)'_\text{pt}$ for $\ell\geq1$ and $\mathcal{C}(\mathfrak{sl}_2,4\ell+2)'_\text{pt}$ for $\ell\geq3$ from \cite[Section 5.5]{DNO} are all self dual.  There are only three self dual categories of the form in (\ref{1}) through (\ref{3}): $\mathcal{C}(\mathfrak{sl}_3,2)'_\text{pt}$, $\mathcal{C}(\mathfrak{sl}_3,3)_A^0$, and $\mathcal{C}(\mathfrak{sl}_3,6)_A^0$.  The first is equivalent to $(\mathcal{C}(\mathfrak{sl}_2,3)'_\text{pt})^\text{rev}$ as noted in Proposition \ref{threedoesnotdividek} (b) and the implications were discussed above, while the second was already discussed as an element of $\mathcal{W}_{\text{Ising}}$.  It remains to show that $\mathcal{C}(\mathfrak{sl}_3,6)_A^0$ is not equivalent to $\mathcal{C}(\mathfrak{sl}_2,11)'_\text{pt}$ since these categories have the same number of simple objects.  But the formula found in \cite[Section 6.4 (1)]{DMNO} and Lemma \ref{centralcharge} imply their central charges are not equal and hence these categories are not equivalent.

\par The last step is to check that none of the categories $\mathcal{C}(\mathfrak{sl}_2,4\ell)_A^0$ are equivalent to the categories in (\ref{1}) through (\ref{3}).  We will do so by comparing central charges.  Lemma \ref{centralcharge} states for $\ell\in\mathbb{Z}_{>0}$
\begin{equation*}
\xi\left(\mathcal{C}(\mathfrak{sl}_2,4\ell)_A^0\right)=\exp\left(\lambda(\ell)\pi i\right)\qquad\text{ where }\qquad\lambda(\ell)=\dfrac{3\ell}{4\ell+2}.
\end{equation*}
and thus $1>\lambda(\ell)>1/2$.  By the inequality in (\ref{tabel}) and the values in (\ref{table}) exceptional equivalences of the form $\mathcal{C}\simeq\mathcal{D}^\text{rev}$ may exist since $\lambda(1)+\lambda_{(\ref{3})}(3)=2$, $\lambda(4)+\lambda_{(\ref{3})}(2)=2$, $\lambda(7)+\lambda_{(\ref{2})}(0)=0$, and $\lambda(3)=\lambda_{(\ref{1})}(1)$.  The first case was covered in Equation  (\ref{4and1}) above.  In the second case not much work is needed since there is a conformal embedding $A_{2,6}\times A_{1,16}\subset E_{8,1}$ which gives the relation
\begin{equation}[\mathcal{C}(\mathfrak{sl}_3,6)][\mathcal{C}(\mathfrak{sl}_2,16)]=[\Vep].\label{last}\end{equation}
The third case was covered in Equation (\ref{2and28}) above.  The last equality is because $\mathcal{C}(\mathfrak{sl}_3,4)'_\text{pt}\simeq\mathcal{C}(\mathfrak{sl}_2,12)^0_A$ by the classification of rank 5 modular tensor categories \cite{rowell}.  Hence
\begin{equation}[\mathcal{C}(\mathfrak{sl}_3,4)][\mathcal{C}(\mathfrak{sl}_3,1)]=[\mathcal{C}(\mathfrak{sl}_2,12)]\label{reallylast}\end{equation}.

\begin{mycorr}
All nontrivial relations between the equivalency classes $[\mathcal{C}(\mathfrak{sl}_2,k)]$ and $[\mathcal{C}(\mathfrak{sl}_3,k)]$ are listed in (\ref{16})--(\ref{reallylast}). 
\end{mycorr}


\section{Future Directions}

\par Since such a wide variety of results are required, the above classification of Witt group relations is a perfect motivation and launching point for considering many related research topics.  It is a nontrivial task to extend the results of Theorems 1 and 2 in the same spirit due to the fact that numerical computations for twists, dimensions, and fusion coefficients for $\mathcal{C}(\mathfrak{g},k)$ become unwieldly for simple Lie algebras $\mathfrak{g}$ of larger rank.  For now, the most reasonable generalization is to the case $\mathfrak{g}=\mathfrak{sl}_p$ with $p$ prime.  In these cases the prime decomposition analagous to Proposition \ref{2dk} when $k\nmid p$ is clear, with each $\mathcal{C}(\mathfrak{sl}_p,k)$ factoring by M\"uger's decomposition into a pointed part (corresponding to the cyclic group $\mathbb{Z}/p\mathbb{Z}$) and its centralizer corresponding to the root lattice in the Weyl alcove.

\par The necessary results which do not extend trivially to $\mathfrak{sl}_p$ are the classification of connected \'etale algebras in $\mathcal{C}(\mathfrak{sl}_p,k)$ and the simplicity of $\mathcal{C}(\mathfrak{sl}_p,k)_A^0$ for connected \'etale algebras of Type D when $p\mid k$.  The proof of simplicity relied on Lemma \ref{mplusone} which gave an explicit formula for the occurence of the central weight $\nu$ in tensor powers of itself.  For $\mathcal{C}(\mathfrak{sl}_p,k)$ with $p>3$ the growth of this fusion coefficient is dramatically non-linear as $k$ increases, but numerical evidence suggests a similar argument can be made.

\par Classifying connected \'etale algebras in $\mathcal{C}(\mathfrak{sl}_n,k)$ (and general $\mathfrak{g}$) is the more interesting obstruction of the two.  The proof of this classification for $\mathfrak{sl}_3$ (Proposition \ref{gannon}) due to Gannon \cite{gannon} is over two decades old as this paper is being written, at a time when the theory of tensor categories was not fully developed.  A novel and promising approach utilizing this modern theory, attributed to Ocneanu can tentatively place an upper bound on the levels $k$ for which exceptional (Type E) modular invariants can exist in $\mathcal{C}(\mathfrak{sl}_n,k)$.

\par One other remaining question is related to the note in Section \ref{threedividesk}: that the group of invertible objects corresponding to corner weights acts on $\mathcal{C}(\mathfrak{sl}_3,3m)$ by rotations of the Weyl alcove.  The formal study of all categorical group actions on a given monoidal category $\mathcal{C}$ is described in Section 2.7 of \cite{tcat}, and requires the computation of Aut$_\otimes(\mathcal{C})$, the category of all monoidal autoequivalences of $\mathcal{C}$.  An explicit characterization of these monoidal autoequivalences for the categories $\mathcal{C}(\mathfrak{g},k)$ is given in \cite{neshtuset}, but currently there are no analogous results for the dyslectic module categories $\mathcal{C}(\mathfrak{g},k)^0_A$ although it should be expected that a similar, if not simpler, characterization should exist.

\bibliographystyle{plain}
\bibliography{bib}

\end{document}